\documentclass[11pt,a4paper]{amsart}
\usepackage{amsmath,amsthm,amsfonts,graphicx,amssymb,amscd}
\usepackage{graphicx}
\usepackage{hyperref}
\newtheorem{thm}{Theorem}[section]
\newtheorem*{jthm}{Theorem}
\newtheorem*{thm14}{Theorem 1.4}

\newtheorem{cor}[thm]{Corollary}
\newtheorem{lem}[thm]{Lemma}
\newtheorem{prop}[thm]{Proposition}
\theoremstyle{definition}
\newtheorem{defn}[thm]{Definition}

\begin{document}

\title{Asymptoticity of grafting and Teichm\"{u}ller rays}
\author{Subhojoy Gupta}
\address{Department of Mathematics, Yale University, New Haven, CT 06520, USA}
\curraddr{Center for Quantum Geometry of Moduli Spaces, Ny Munkegade 118, DK 8000 Aarhus C, Denmark. }
\email{sgupta@qgm.au.dk}
\date{December 23, 2012}

\begin{abstract}
We show that any grafting ray in Teichm\"{u}ller space determined by an arational lamination or a multi-curve is (strongly) asymptotic to a Teichm\"{u}ller geodesic ray. As a consequence the projection of a generic grafting ray to moduli space is dense. We also show that the set of points in Teichm\"{u}ller space obtained by integer ($2\pi$-) graftings on any hyperbolic surface projects to a dense set, which implies that complex projective surfaces with any fixed Fuchsian holonomy are dense in moduli space. 
\end{abstract}

\maketitle
\tableofcontents
\section{Introduction}
A \textit{complex projective structure} on a surface $S_g$ of genus $g$ is an atlas of charts to $\mathbb{C}P^1$ such that the transition maps are in $PSL_2(\mathbb{C})$, and as this also determines a marked conformal structure, the space $\mathcal{P}(S_g)$ of such structures forms a bundle over Teichm\"{u}ller space $\mathcal{T}_g$.  In particular, a hyperbolic structure on a surface can be thought of as a complex projective structure with Fuchsian (or real) holonomy, and the operation of \textit{projective grafting} on a simple closed curve deforms such a complex projective structure by inserting a projective annulus along a geodesic representative of that curve. By taking limits, this procedure extends to geodesic laminations and gives a geometric parametrization of $\mathcal{P}(S_g)$ (see for example \cite{KamTan}, \cite{Tan}, \cite{McM}). In this paper we shall consider \textit{conformal grafting rays} which are the image in $\mathcal{T}_g$ of such deformations of complex projective structures, and establish a (strong) asymptoticity with Teichm\"{u}ller geodesics (Theorem 1.1) that is used to show a density result concerning the set of complex projective structures with Fuchsian holonomy (Theorem 1.4).\\

The conformal grafting ray determined by a pair $(X,\lambda)$ of a hyperbolic surface $X$ and a measured geodesic lamination $\lambda$ shall be denoted by $gr_{t\lambda}X$, and is a real-analytic one-parameter family of conformal structures obtained, roughly speaking, by cutting along $\lambda$ on $X$  and inserting a euclidean metric whose width increases along the ray (a more precise description is given later). Associated to a pair $(X,\lambda)$ there is also a \textit{Teichm\"{u}ller geodesic ray} starting from $X$ and in the ``direction" $\lambda$ (see Definition \ref{defn:teichray}). This paper establishes a strong asymptoticity between the two.\\

Two rays $\Theta$ and $\Psi$ in $\mathcal{T}_g$ are said to be asymptotic if the \textit{Teichm\"{u}ller distance} (defined in \S2) between them goes to zero, after reparametrizing if necessary. More concisely, $\lim\limits_{t\to\infty}\inf\limits_{Z\in \Theta}d_{\mathcal{T}}(Z, \Psi(t)) =0$. We shall prove here that:

\begin{thm}\label{thm:thm1} Let $X\in \mathcal{T}_g$ and let $\lambda \in \mathcal{ML}$ be such that $\lambda$ is arational, or a multicurve. Then there exists a $Y\in \mathcal{T}_g$ such that the grafting ray determined by $(X,\lambda)$ is asymptotic to the Teichm\"{u}ller ray determined by $(Y,\lambda)$.
\end{thm}

Here a measured lamination $\lambda$ is said to be \textit{arational} when it is both \textit{maximal} (complementary regions are all triangular) and \textit{irrational} (has a single minimal component that is not a closed geodesic). Such laminations are in fact of full measure in $\mathcal{ML}$ - we refer to \S 2 for a fuller discussion of the structure theory of geodesic laminations. In a sequel to this paper (\cite{Gup2}) we generalize Theorem 1.1 to the case of a general lamination. \\

The following are immediate corollaries of Theorem 1.1 and the work of Masur (\cite{Mas}, \cite{MasErg}):
\begin{cor}
Let  $X, Y$ be any two hyperbolic surfaces and let $\lambda$ be a maximal uniquely-ergodic lamination. Then the grafting rays determined by $(X,\lambda)$ and $(Y,\lambda)$ are asymptotic.
\end{cor}

\begin{cor}\label{cor:cor1}
For every $X\in \mathcal{T}_g$ and almost every $\lambda \in \mathcal{ML}$ in the Thurston measure, the projection of the grafting ray determined by $(X,\lambda)$ is dense in moduli space $\mathcal{M}_g$.
\end{cor}

Let $\mathcal{S}$ be the set of simple closed geodesic multicurves on $S_g$. As a further application of the techniques in the proof of Theorem \ref{thm:thm1} we show: 

\begin{thm}\label{thm:thm1.5}
For any $X\in\mathcal{T}_g$, the set of integer graftings $\{\pi(gr_{2\pi\gamma}X)\vert$ $\gamma\in \mathcal{S}\}$ is dense in $\mathcal{M}_g$.
\end{thm}
\index{Goldman}
Using the work of Goldman in \cite{Gold}, this implies the following:

\begin{thm}\label{thm:thm2}
Let $\mathcal{P}_\rho$ be the set of complex projective structures on a surface $S_g$ with a  fixed holonomy $\rho\in Rep(\pi_1(S_g),PSL_2(\mathbb{C}))$. Then for any Fuchsian representation $\rho$, the projection of $\mathcal{P}_\rho$ to $\mathcal{M}_g$ has a dense image.
\end{thm}
\index{Faltings}
In \cite{Falt} Faltings had first conjectured that this projection of the set $\mathcal{P}_\rho$ is infinite, and this result can be thought of as the strongest possible affirmation of that.\\

The asymptotic behavior (as in Theorem \ref{thm:thm1}) of two Teichm\"{u}ller rays with respect to the Teichm\"{u}ller metric is well known (see \cite{Mas}, \cite{Iv},\cite{LenMas}). In \cite{Mas} Masur proved that if $\lambda$ is uniquely ergodic, then for any two initial surfaces $X,Y\in \mathcal{T}_g$  the Teichm\"{u}ller rays determined by $(X,\lambda)$ and $(Y,\lambda)$ are asymptotic. The comparison of grafting rays and Teichm\"{u}ller rays has been less explored, however a recent result along these lines (see also \cite{DiazKim}) is the following ``fellow-traveling" result in \cite{RafiDumasChoi}:

\begin{jthm}[Choi-Dumas-Rafi]
For any $X\in \mathcal{T}_g$ and any unit length lamination $\lambda$, the grafting ray determined by $(X,\lambda)$ and the Teichm\"{u}ller ray determined by $(X,\lambda)$ are a bounded distance apart, where the bound depends only on the injectivity radius of $X$.
\end{jthm}

Theorem 1.1 makes a finer but less uniform comparison involving the stronger notion of asymptoticity defined above.\\

A crucial difference between grafting and Teichm\"{u}ller rays is that the latter are lines of a flow, whereas the former are not:
\begin{equation*}
gr_{t\lambda} \circ gr_{s\lambda} \neq gr_{(t+s)\lambda} 
\end{equation*}

Our results show that if one waits till $t$ is sufficiently large, however, it approximates the Teichm\"{u}ller geodesic flow. We aim to discuss this and a more quantitative version of Theorem \ref{thm:thm1} in forthcoming work.\\\

The proof of Theorem 1.1 is achieved by constructing quasiconformal maps of small dilatation from sufficiently large grafted surfaces along the grafting ray, to  singular flat surfaces that lie along a common Teichm\"{u}ller ray. It involves a comparison of the \textit{Thurston metric}, a hybrid of a hyperbolic and euclidean metric underlying a complex projective surface on one hand, and a singular flat metric induced by a holomorphic quadratic differential on the other. The case when the lamination $\lambda$ is arational is handled in \S 4, and the case when the lamination is a multicurve is dealt with in \S 5. The proof of Theorem \ref{thm:thm1.5} in \S6 is obtained by careful approximation of an arational lamination with multicurves. An outline of the strategy of the proofs of both theorems are provided in \S 3. \\

\textbf{Acknowledgements.} This paper represents part of my Ph.D thesis at Yale University. I wish to thank my advisor Yair Minsky for his generous help and patient guidance during this project. 

\section{Preliminaries}

\textbf{Teichm\"{u}ller space $\mathcal{T}_g$.}
For a closed orientable genus-$g$ surface $S_g$,  the Teichm\"{u}ller space $\mathcal{T}_{g}$ is the space of marked conformal (or equivalently, complex) structures on $S_g$ with the equivalence relation of isotopy (see \cite{ImTan}, \cite{Hubb}  for a treatment of the subject). Note that although for this article, we assume the surfaces have no punctures, the results still hold for the punctured case with a slight modification of the arguments.\newline

The distance between two points $X$ and $Y$ in $\mathcal{T}_g$ in the \textit{Teichm\"{u}ller metric} is defined to be
\begin{equation*}
d_{\mathcal{T}}(X,Y) = \frac{1}{2}\inf\limits_{f} \ln K_f
\end{equation*}
where $f:X\to Y$ is a quasiconformal homeomorphism, and $K_f$ is its quasiconformal dilatation. The infimum is realized by the \textit{Teichm\"{u}ller map} between the surfaces.\newline

A thorough discussion of the definitions of a quasiconformal map and its dilatation (also referred to as its quasiconformal \textit{distortion}) is provided in Appendix A. It suffices to point out here that roughly speaking, a quasiconformal map takes infinitesimal circles on the domain to infinitesimal ellipses, and the dilatation is a measure of the maximum eccentricity of the image ellipses. The `difference' or distance between two conformal structures is then measured in terms of the dilatation of the \textit{least} distorted quasiconformal map between them.\newline

A consequence of the above definition is that if there exists a map $f:X\to Y$ which is $(1 + O(\epsilon))$-quasiconformal (i.e, $K_f = 1 + O(\epsilon)$), then 
\begin{equation*}
d_{\mathcal{T}}(X,Y) = O(\epsilon).
\end{equation*}
\textit{Notation.} Here, and throughout this article, $O(\alpha)$ refers to a quantity bounded above by $C\alpha$ where $C>0$ is some constant depending only on genus $g$ (which remains fixed), the exact value of which can be determined \textit{a posteriori}.

\begin{defn}\label{defn:ac}
For the ease of notation, an \textit{almost-conformal map} shall refer to a map which is $(1+ O(\epsilon))$-quasiconformal.
\end{defn}

\textbf{Hyperbolic surfaces and geodesic laminations.}
Any conformal structure on a surface of genus $g\geq 2$ has a unique hyperbolic structure (a Riemannian metric of constant negative curvature $-1$) in its conformal class via uniformization. Thus Teichm\"{u}ller space is, equivalently, the space of marked hyperbolic structures and this gives rise to a rich interaction between two-dimensional hyperbolic geometry and complex analysis.\newline

A \textit{geodesic lamination} on a hyperbolic surface is a closed subset of the surface which is a 
union of disjoint simple geodesics. A \textit{maximal} lamination is a geodesic lamination such that any component of its complement is an ideal hyperbolic triangle. There is a rich structure theory of geodesic laminations (see \cite{CassBl} for example): in particular, any geodesic lamination $\lambda$ is a  disjoint union of sublaminations
\begin{equation}\label{eq:lamdec}
\lambda = \lambda_1\cup \lambda_2 \cup\cdots \lambda_m \cup \gamma_1 \cup \gamma_2\cdots \gamma_k
\end{equation}
where $\lambda_i$s are minimal components (with each half-leaf dense in the component) which consist of uncountably many geodesics (a Cantor set cross-section) and the $\gamma_j$s are isolated geodesics.\newline

A \textit{measured} geodesic lamination is equipped with a transverse measure $\mu$, that is a measure on arcs transverse to the lamination which is invariant under sliding along the leaves of the lamination. It can be shown that for the support of a measured lamination the isolated leaves in (\ref{eq:lamdec}) above are weighted simple closed curves (ruling out the possibility of isolated geodesics spiralling onto a closed component).  We call a lamination \textit{arational} if every simple closed curve intersects it, and it can be shown that for a \textit{measured} lamination this condition is equivalent to being maximal and \textit{irrational}, which is that it consists of a single minimal component and no isolated leaves. A measured lamination is \textit{uniquely-ergodic} if such a measure is unique. A maximal, uniquely ergodic lamination is necessarily arational. The set of weighted simple closed curves is dense in $\mathcal{ML}$, the space of measured geodesic laminations equipped with the weak-* topology. $\mathcal{ML}$ also has a piecewise-linear structure and a corresponding \textit{Thurston measure}.
 \newline

\textbf{Train tracks.}
A \textit{train-track} on a surface is a graph with a labeling of incoming and outgoing half-edges at every vertex, and an assignment of (non-negative) weights to the edges (or \textit{branches}) that are \textit{compatible}, such that at every vertex, the sum of the weights of the incoming edges is equal to the sum of the edges of the outgoing edges. This provides a convenient combinatorial encoding of a lamination (see, for example, \cite{FLP} or \cite{Thu0}) - in particular, for an assignment of \textit{integer} weights, one can place a number of strands along each branch equal to the weight, and the compatibility condition ensures that these can be ``hooked" together to form a multicurve.\\

\textbf{Quadratic differentials and Teichm\"{u}ller rays.}
	Any measured geodesic lamination corresponds to a unique \textit{measured foliation} of the surface, obtained by `collapsing' the complementary components. Conversely, any measured foliation can be `tightened' to a geodesic lamination. The space of measured foliations $\mathcal{MF}$ is homeomorphic to $\mathcal{ML}$ via this correspondence (see, for example, \cite{Kap}).\newline
	A holomorphic quadratic differential $\phi$ (see \cite{Streb} for a treatment of the subject) is locally of the form $\phi(z)dz^2$  where $\phi(z)$ is a holomorphic function, and has a vertical foliation given by the level sets of $Im(\int\limits_0^z\sqrt{\phi(z)}dz)$, which has singularities at the zeroes of $\phi(z)$. These singularities are also the cone-points in the singular flat metric given by $\lvert \phi(z)\rvert \lvert dz\rvert^2$.\newline
	For a hyperbolic surface $X$, the map that assigns the vertical measured foliation $\mathcal{F}_v(\phi)$ to a quadratic differential $\phi(X)$ defines a homeomorphism between the space of quadratic differentials $Q(X)$ and $\mathcal{MF}$ (\cite{HubbMas}). Composed with the previous correspondence between measured laminations and foliations, we get a homeomorphism
\begin{equation*}
q_L:Q(X) \to \mathcal{ML}
\end{equation*}
\begin{defn}\label{defn:teichray}
A \textit{Teichm\"{u}ller ray} from a point $X$ in $\mathcal{T}_g$ and in a direction determined by a holomorphic quadratic differential $\phi$ is a subset $\{X_t\}_{t\geq 0}$ of $\mathcal{T}_g$ where $X_t$ is obtained by starting with the surface $X$ with the horizontal and vertical foliations $\mathcal{F}_h(\phi)$ and $\mathcal{F}_v(\phi)$ and scaling the horizontal foliation by a factor of $e^t$ and the vertical foliation by a factor of $e^{-t}$. 
\end{defn}
This ray is geodesic in the Teichm\"{u}ller metric (see, for example,  \cite{ImTan}).\newline

By the above correspondence, we define a \textit{Teichm\"{u}ller ray determined by a pair $(X,\lambda) \in \mathcal{T}_g\times \mathcal{ML}$} to be the the Teichm\"{u}ller ray starting from $X$ in the direction determined by $q_L^{-1}(\lambda)$.\\

\textbf{Complex projective structures and grafting.}
An excellent exposition of this material can be found in \cite{Dum0}, the following is a brief summary.\newline

A \textit{complex projective structure} on a surface $S_g$  is a pair $(dev, \rho)$ where $dev:\tilde{S_g}\to \mathbb{C}P^1$ is a local homeomorphism which is the \textit{developing} map of its universal cover and $\rho:\pi_1(S_g)\to PSL_2(\mathbb{C})$ is the \textit{holonomy} representation that satisfies
\begin{equation*}
dev\circ \gamma = \rho(\gamma) \circ dev
\end{equation*}
for each $\gamma \in \pi_1(S_g)$.\newline

Equivalently, a complex projective structure is a collection of charts on the surface to $\mathbb{C}P^1$ such that overlapping charts differ by a projective (or M\"{o}bius) transformation. Since M\"{o}bius transformations are holomorphic, any complex projective structure has an underlying conformal structure. This gives rise to a forgetful map from the space of complex projective structures $p:\mathcal{P}(S_g)\to \mathcal{T}_g$.\newline

A hyperbolic surface arises from a \textit{Fuchsian} representation $\rho:\pi_1(S_g)\to PSL_2(\mathbb{R})$ and thus has a canonical complex projective structure.
\textit{Grafting}, introduced by Thurston (see \cite{KamTan}, \cite{Tan},\cite{ScanWolf}, \cite{Dum1} for subsequent development) can be thought of as a way to deform the Fuchsian complex projective structure.\newline

In the case of grafting along a simple closed geodesic $\gamma$ with transverse measure (weight) $s$, the process can be described as follows:\newline
Embed the universal cover $\tilde{X}$  of the hyperbolic surface $X$ as the equatorial plane in the ball-model of $\mathbb{H}^3$. The hyperbolic Gauss map to the $\partial \mathbb{H}^3 = \mathbb{C}P^1$ (the inverse of the \textit{nearest point rectraction}) provides the developing map of the Fuchsian structure. The curve $\gamma$ lifts to a collection of geodesic lines on this plane. Now we bend along these lifts equivariantly by an angle $s$ such that one gets a convex pleated plane. Via the Gauss map, this corresponds to inserting crescent-shaped regions of angle $s$ along the images of the lifts of $\gamma$ (in Figure 10 in \S \ref{fig:graft1} this is shown in the upper half-plane model where the imaginary axis is a lift of $\gamma$). This defines the image of the developing map of the new projective structure. Of course, for an angle $s\geq 2\pi$, the image wraps around $\mathbb{C}P^1$, and the developing map is not injective. The new holonomy is the $PSL_2(\mathbb{C})$-respresentation  compatible with this new developing map in the sense described above.\newline

Grafting for a \textit{general} measured lamination $\lambda$ is defined by taking the limit of a sequence of approximations of $\lambda$ by weighted simple closed curves, that is, a sequence $s_i\gamma_i \to \lambda$ in $\mathcal{ML}$. For each $s_i\gamma_i$, one can form a pleated plane by equivariant bending, and a corresponding complex projective structure, as above.  It follows from the foundational work of Epstein-Marden (\cite{EpMar}) and Bonahon (\cite{Bon}) that 
\begin{center}

\textit{the convex pleated planes converge in the Gromov-Hausdorff sense}

\end{center}
which implies that
\begin{center}

\textit{the developing maps converge uniformly on compact sets}

\end{center}
and
\begin{center}

\textit{the corresponding holonomy representations converge algebraically}

\end{center}
so, in particular, there is a limiting complex projective structure on the surface $S_g$.\\

 It was an observation of Thurston (see \cite{KamTan},\cite{Tan}) that the map
\begin{equation*}
(X,\lambda)\mapsto Gr_{\lambda}X
\end{equation*}
is a homeomorphism of $ \mathcal{T}_g\times \mathcal{ML}$ to $\mathcal{P}_g$, where $Gr_{\cdot}\cdot$ refers to the \textit{projective} surface obtained by the above operation.\newline

In this paper we are concerned with \textit{conformal grafting} $gr:  \mathcal{T}_g\times \mathcal{ML}\to  \mathcal{T}_g$, where we consider only the conformal structure underlying the complex projective structure (so that $gr = p\circ Gr$ where $p:\mathcal{P}_g\to \mathcal{T}_g$ is the usual projection).  It is known that for any fixed lamination $\lambda$, the grafting map ($X\mapsto gr_{\lambda}X$) is a self-homeomorphism of $\mathcal{T}_g$ (\cite{ScanWolf}).\\

\textbf{Thurston metric.} 
A complex projective structure on a surface $S_g$ determines a pair of the developing map $dev$  from the universal cover $\widetilde{S_g}$ to $\mathbb{C}P^1$ and the holonomy representation $\rho$.
There is a canonical stratification of the image of $dev$ on $\mathbb{C}P^1$ (\cite{KulPink}), and in particular, one can speak of maximally embedded round disks there. One can in fact recover the locally convex pleated plane (that one had in the bending description above) by taking an envelope of the convex hulls (domes) over these disks.

\begin{defn}\label{defn:projmet}
The \textit{projective metric} on the universal cover $\widetilde{S_g}$ is defined by pulling back the Poincar\'{e} metric on the maximal disks, via the developing map $dev$. The projective metric descends to the \textit{Thurston metric} on the surface, under the quotient by the action of $\rho(\pi_1(S_g))$.
\end{defn}

\begin{figure}
  \centering
  \includegraphics[scale=0.45]{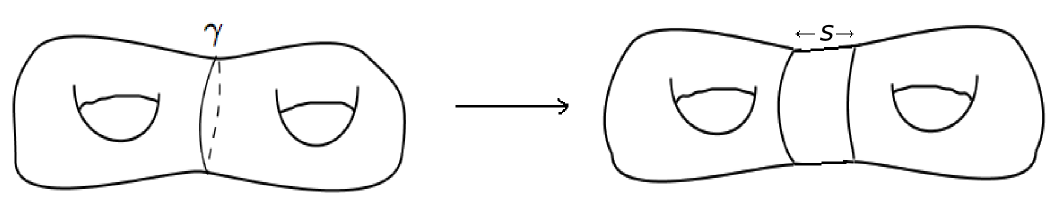}\\
  \caption{The grafting map along the weighted curve $s\gamma$. }
\end{figure}

When $\gamma$ is a simple closed curve, the equivariant collection of $s$-crescents obtained by bending the lifts of $\gamma$ on the equatorial plane in $\mathbb{H}^3$  (see the previous section) descend to a annulus inserted at the closed geodesic $\gamma$. So in the Thurston metric on  $gr_{s\gamma} X$, the inserted annulus is flat (euclidean), of width $s$, and the rest of the surface remains hyperbolic.\newline

For a general lamination, the Thurston metric on $gr_{\lambda}X$  is the limit of the Thurston metrics on $gr_{s_i\gamma_i}X$ where $s_i\gamma_i\to \lambda$ is an approximating sequence of weighted simple closed curves (see\cite{ScanWolf}). The length in the Thurston metric on $gr_{\lambda}X$ of an arc $\tau$  intersecting $\lambda$, is its hyperbolic length on $X$ plus its transverse measure.

\begin{defn}
A \textit{grafting ray} from a point $X$ in $\mathcal{T}_g$ in the `direction' determined by a lamination $\lambda$ is the $1$-parameter family $\{X_t\}_{t\geq 0}$ where $X_t = gr_{t\lambda}(X)$.
\end{defn}

 \section{An overview of the proofs}
The proofs of both Theorem \ref{thm:thm1} and Theorem \ref{thm:thm2} involve understanding the geometry of the grafted surfaces $X_t = gr_{t\lambda} X$ along the grafting ray determined by a pair $(X,\lambda)$, for large $t$. By the definition of grafting, as described in \S2, these surfaces carry a conformal metric which is euclidean on the grafted region and hyperbolic elsewhere. A convenient way to picture this is to consider a thin ``train-track" neighborhood  $\mathcal{T} \subset X$ that contains the lamination $\lambda$. The intuition is  that as one grafts, the subsurface $\mathcal{T}$ widens in the transverse direction (along the ``ties" of the train-track), and conformally approaches a union of wide euclidean rectangles.\\

The complement $X\setminus \mathcal{\lambda}$ is unaffected by grafting: in the case when $\lambda$ is \textit{arational}, it comprises finitely many hyperbolic ideal triangles (the number depends only on genus) and for $\lambda$ non-arational this complement  might consist of ideal polygons or  subsurfaces with moduli. The former case therefore is simpler and allows for more explicit constructions, and we focus on that first. For the latter case, this paper shall deal with the case when these subsurfaces have boundary components consisting of closed curves, and the general case is deferred to a subsequent paper. In either case, the underlying intuition is that this  complementary hyperbolic part becomes negligible compared to the euclidean part of the Thurston metric, for a sufficiently large grafted surface. (This has been exploited before in (\cite{Dum2}.)

\subsection{The arational case}
\subsection*{The surface $\hat{X_t}$.}

For $\lambda$ arational, one can consider the associated (singular) transverse horocyclic foliation which we denote by $\mathcal{F}$. This is obtained as follows:\\
Since the lamination $\lambda$ is maximal, it lifts to the universal cover of the surface to give a tessellation of the hyperbolic plane $\mathbb{H}^2$ by ideal hyperbolic triangles. Each ideal hyperbolic triangle has a partial foliation by horocyclic arcs belonging to horocycles tangent to each of the three ideal vertices. It can be shown (see \cite{Thu}) that by some slight modification this can be extended across the ideal triangles and to the central region (missed by the horocyclic arcs) to form a foliation $\mathcal{F}$ of the surface which is transverse to $\lambda$, with ``$3$-prong" singularities in the center of each ideal triangle, and the foliation being $C^1$ away from the singularities (the leaves are flowlines of a Lipschitz vector field).
\begin{figure}[h]
  \centering
  \includegraphics[scale=0.3]{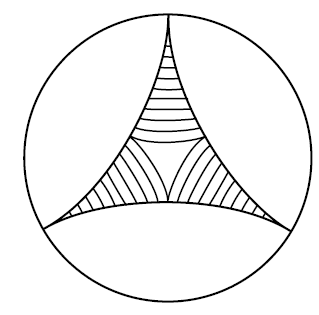}\\
  \caption[The partial horocyclic foliation of an ideal hyperbolic triangle. ]{The partial horocyclic foliation of an ideal hyperbolic triangle. This can be modified to extend it to a singular foliation. }
\end{figure}

This transverse foliation $\mathcal{F}$ persists under grafting along $\lambda$, with the leaves getting longer along a grafting ray.\newline

Then there is an associated Riemann surface $\hat{X}_t$ with a singular flat metric obtained by collapsing the ideal triangle components of $X\setminus \lambda$ along the leaves of $\mathcal{F}$, but preserving the transverse measure. The singularities of the flat metric on $\hat{X_t}$ are $3$-pronged conical singularities that arise by collapsing the central (unfoliated) region of each ideal triangle in the complement of the lamination.\\

There are a couple of ways one can think of the collapsed surface: one is to think of an explicit collapsing map (see \cite{CassBl}) that ``blows up" the lamination (that is locally a Cantor set cross interval) by a function that is the Cantor function on each transverse cross section. The other is to think of $\hat{X}_t$ as a singular flat surface obtained by gluing up euclidean rectangles with the same combinatorics of the gluing as dictated by the structure of the lamination (or equivalently the corresponding train-track $\mathcal{T}$) on $X_t$. \\

The singular flat surfaces $\hat{X_t}$ acquire a horizontal foliation that is measure-equivalent to $\mathcal{F}$, and a vertical foliation that is measure-equivalent to $t\lambda$.  Hence as $t$ varies they lie on a common Teichm\"{u}ller ray. The ``collapsing map" itself is far from being quasiconformal (it is not even a homeomorphism), and main idea behind the proof of Theorem \ref{thm:thm1} is to use the additional grafted region to ``diffuse out" the collapsing to get a \textit{quasiconformal} map from $X_t$ to $\hat{X_t}$, which is moreover almost-conformal (see Definition \ref{defn:ac}).\\

\textbf{Outline of the proof (arational case):}

\textit{Step 1. The decomposition of $X_t$.} One first decomposes the grafted surface into rectangles and pentagonal pieces  (\S 4.1) that essentially make up the train track neighborhood $\mathcal{T}$, and a slight thickening of the truncated ideal triangles in its complement, respectively. Lemmas \ref{lem:thinrect} and \ref{lem:widerect} are concerned with the dimensions of the resulting pieces.\\

\begin{figure}[h]
  \centering
  \includegraphics[scale=0.6]{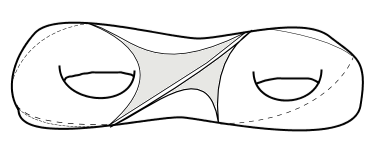}\\
  \caption[A partial picture of a maximal lamination $\ldots$]{A partial picture of a maximal lamination, with two truncated ideal triangles in the complement shown shaded. }
\end{figure}

\textit{Step 2. Mapping the pieces.} Next, one constructs quasiconformal maps that map the pieces in the decomposition to (singular) flat regions, by mapping the leaves of $\mathcal{F}$ in a suitable manner. The rectangular pieces of the grafted surface that ``carries" the lamination need a finite approximation argument, which is carried out in \S 4.3, and culminates in Lemma \ref{lem:rectModel}. The maps for the pentagonal pieces (section \S 4.4) are constructed by first constructing maps of ``truncated sectors" . Much of these constructions depend on explicit constructions of $C^1$ maps of controlled dilatation that one can build between various hyperbolic or euclidean regions, which are compiled in \S 4.2. The assumption of $C^1$-regularity is justified by the corresponding regularity of the Thurston metric (see \S 4.3) and the horocyclic foliation $\mathcal{F}$.\\

\begin{figure}[h]
  \centering
  \includegraphics[scale=0.45]{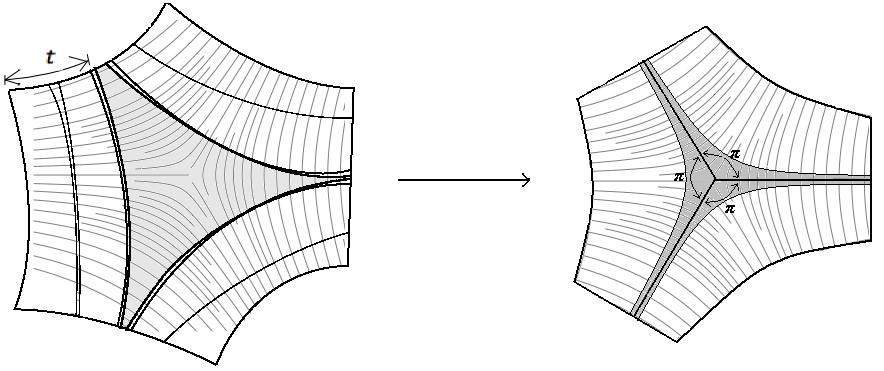}\\
  \caption[In Step 2, the maps of the pieces assemble to  give a quasiconformal map $\ldots$]{In Step 2, the maps of the pieces assemble to give a quasiconformal map of the portion of $X_t$ shown on the left to the singular flat region on $\hat{X_t}$ shown on the right. The (shaded) hyperbolic region is taken to a neighborhood of the central ``tripod".}
\end{figure}

\textit{Step 3. Gluing the maps of the pieces.} To assemble the maps of the pieces to a map of the grafted surface $X_t$ to the singular flat surface $\hat{X_t}$,  they need to be adjusted on the boundary. This is possible by an additional property of the maps called \textit{almost-isometry} (Definition \ref{defn:aisom}) which allows this adjustment to be made maintaining the almost-conformality (see Lemma \ref{lem:qcnoskew} in \S 4.2). At this stage one has a quasiconformal map that is almost-conformal for \textit{most} of the surface (Lemma \ref{lem:Map}).\\

\textit{Step 4. Adjusting to an almost-conformal map.} The quasiconformal map from Lemma \ref{lem:Map} is then adjusted to be almost-conformal \textit{everywhere}. This relies on the fact that the regions of no control of quasiconformal distortion are contained in portions of the surface surrounded by annuli of large modulus (Lemma \ref{lem:largemod}), and a technical lemma on quasiconformal extensions  (Lemma \ref{lem:qclem}) whose proof we defer to the Appendix.

\subsection{The multicurve case}

In contrast to the case of an arational (and therefore maximal) lamination, for a general lamination the complementary subsurface $X\setminus \mathcal{T}$ might contain an ideal polygon or non-simply-connected subsurface with a non-trivial parameter space (moduli) of conformal structures. Here $ \mathcal{T}$ is a train-track neighborhood of the lamination, as before.  This subsurface remains is unaffected by the procedure of grafting.\newline

In this article we shall consider the case when $\lambda$ is a multicurve, and the general case shall be handled in a subsequent article (\cite{Gup2}). The following is the outline of the argument for the special case when $\lambda$ is a single non-separating simple closed geodesic $\gamma$:\\

The surface $X_t$ appearing along the grafting ray has a euclidean cylinder of length $t$ inserted at $\gamma$. As $t\to \infty$, the surface has a conformal limit  $X^\infty$, which is an ``infinitely grafted" surface $X^\infty$ obtained by gluing semi-infinite cylinders on the two boundary components of $X_t \setminus \gamma$. By abuse of notation we shall also think of $X^\infty$ as the \textit{compact} Riemann surface with two marked points obtained by filling in the punctures.\\

\index{Strebel}\index{quadratic differential!meromorphic}
To find the singular flat surface $\hat{X_t}$ to map the grafted surface $X_t$ to, one starts by defining the singular flat surface  $Y^\infty$ that appears as the limit of the Teichm\"{u}ller ray, that by the purported asymptoticity, will be conformally equivalent to $X^\infty$. This is obtained by prescribing a meromorphic quadratic differential on $X^\infty$.\\

\begin{figure}
  \centering
  \includegraphics[scale=0.35]{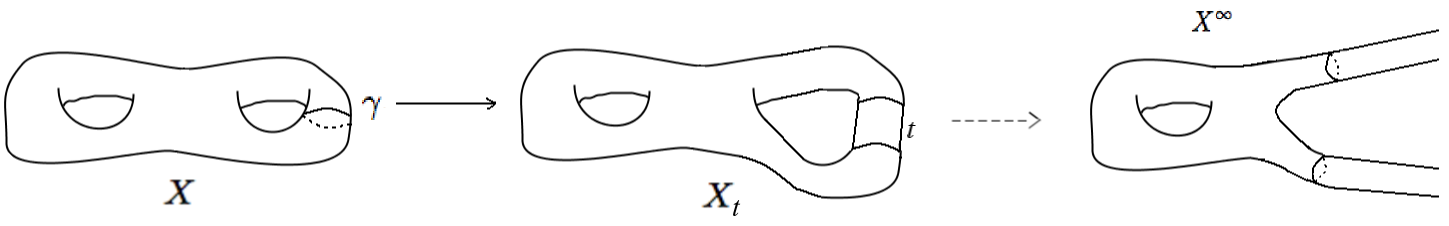}\\
  \caption[Grafting along a simple closed curve]{The surface $X^\infty$ appears as a conformal limit as one grafts along a simple closed non-separating curve $\gamma$.}
\end{figure}

Namely, by a theorem of Strebel (Proposition \ref{prop:streb}), there is a meromorphic quadratic differential on $X^\infty$ with two double poles and closed horizontal trajectories, that induces a singular flat metric which makes the surface isometric to two semi-infinite euclidean cylinders glued along the boundary (we call this $Y^\infty$ to distinguish this from $X^\infty$, though they are conformally equivalent).\\
 
The singular flat surface $\hat{X_t}$ is then obtained by truncating these infinite cylinders of $Y^\infty$ at some ``height" and gluing them along the truncating circles, and this lies along the Teichm\"{u}ller ray determined by $\lambda$. It only remains to adjust the conformal map between $X^\infty$ to $Y^\infty$ to an \textit{almost-conformal} map between $X_t$ and $\hat{X_t}$ - this is possible because a conformal map ``looks" affine at small scales, or (in this case) like an isometry far out a cylindrical end (see Lemma \ref{lem:conf1}).

\subsection{Idea of the proof of Theorem \ref{thm:thm2}:}
When the lamination $\lambda$ is arational at least one of the weights on its train-track representation is irrational. As described at the beginning of \S3, the train track neighborhood $\mathcal{T}$ of $\lambda$ widens along the grafting ray, and a typical rectangular piece (corresponding to a branch of the train track) looks more and more euclidean, with its euclidean width at time $t$ equal to the initial weight times $t$.\\

A key observation is that since the switch conditions for the train-track reduce to a system of linear equations with integer coefficients, there is always an assignment of \textit{integer} weights that approximate those of the arational lamination, such that for each branch the integer weight is within a bound that depends only on the genus (Lemma \ref{lem:linalg}). For sufficiently large $t$, this difference is small in proportion to the entire width, and this allows the construction of an almost-conformal map from a surface along the grafting ray to a surface obtained by grafting along the multi-curve corresponding to the integer solution (Lemma \ref{lem:close}).\\

This construction together with a choice of $\lambda$ that provides a \textit{dense} grafting ray (Corollary \ref{cor:cor1} of Theorem \ref{thm:thm1}) shows that integer graftings are dense in moduli space (Theorem \ref{thm:thm1.5}), and Theorem \ref{thm:thm2} follows as they also preserve the Fuchsian holonomy.

\section{The arational case}
In this section we prove the following proposition, a special (and generic) case of Theorem \ref{thm:thm1}:

\begin{prop}\label{prop:maxcase}
Let $X\in \mathcal{T}_g$ and let $\lambda$ be an arational (ie. maximal and irrational) geodesic lamination. Then there exists a $Y\in \mathcal{T}_g$ such that the grafting ray determined by $(X,\lambda)$ is asymptotic to the Teichm\"{u}ller ray determined by $(Y,\lambda)$.
\end{prop}

An outline of the proof was provided in \S 3.1, and we refer to that section for the notation used here.

\subsection{A train-track decomposition of the grafted surface}
We begin by constructing a subsurface $\mathcal{T}_{\epsilon}\subset X$ containing the arational lamination $\lambda$, that is further decomposed into rectangles. This can be thought of as a physical realization of a train track carrying $\lambda$, and the rectangles correspond to the branches of the train-track. 

\subsubsection{The return map}

From \S3.1 recall that there is a \textit{horocyclic foliation} $\mathcal{F}$ with $C^1$ leaves transverse to $\lambda$, obtained by integrating the Lipschitz line field along the horocylic arcs of the ideal triangles in its complement. Choose an oriented segment $\tau$ from a leaf of $\mathcal{F}$ away from its singularities, such that the endpoints of $\tau$ are on leaves of the lamination $\lambda$ which are isolated on the side away from $\tau$. We shall choose $\tau$ so that its hyperbolic length is small enough, depending on $\epsilon$ - this shall be spelled out in Lemma \ref{lem:thinrect}.\newline

In what follows we use the first return map of $\tau$ to itself (following leaves of $\lambda$) to form a collection of rectangles with vertical geodesic sides, and horizontal sides lying on $\tau$. This is similar to the standard method constructing the suspension of interval exchange transformations,  and in particular to the decomposition in \cite{Mas} where it is done for a quadratic differential metric and the associated foliations.\newline

The outcome of this decomposition is a collection of rectangles $R_1,R_2,\ldots R_n$ on the surface $X$ whose union contains the lamination $\lambda$, and truncated ideal triangles $T_1,T_2,\ldots T_m$ in their complement.\\

\begin{figure}
  \centering
  \includegraphics[scale=0.6]{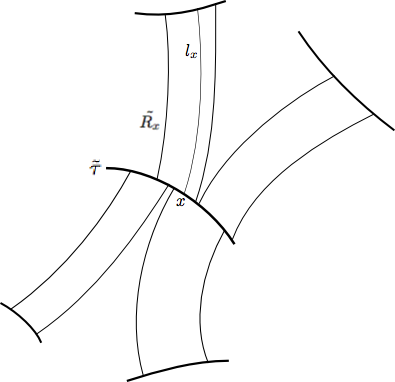}\\
  \caption{The lift of $\lambda$ to the universal cover is decomposed into  rectangular ``flowboxes" by the lifts of the transversal arc $\tau$.}
\end{figure}

For convenience, we lift the lamination to $\tilde{\lambda}$ on the universal cover $\tilde{X}$ of the surface, and pick $\tilde{\tau}$, a choice of lift of the arc $\tau$. Pick any $x\in \tilde{\lambda} \cap \tilde{\tau}$. The fact that $\lambda$ is minimal implies that any half-leaf of $\tilde{\lambda}$ emanating from $x$ intersects another lift of $\tau$. Consider such a segment $l_x$ of a leaf of $\tilde{\lambda}$ between the two lifts of $\tau$. Each point of $l_x$ is contained in a ``flow-box", a rectangle where the lamination $\tilde{\lambda}$ is embedded as $K \times J$, for a Cantor set $K$ and an interval $J$. Using the compactness of $l_x$, one can take a finite subcover by these flow-boxes,  and find a single rectangular flow-box $\tilde{R_x}$ between the two lifts of $\tau$ that contains the entire segment $l_x$. This intersects $\tilde{\tau}$ in an interval $\tilde{I_x}$, and by using the compactness of $\tilde{\tau}$, one can find a finite subcover $\tilde{I_1},\ldots \tilde{I_n}$, and hence finitely many rectangles $\tilde{R_1},\ldots \tilde{R_n}$ that contain all the leaf segments $l_x$ where $x\in \tilde{\lambda} \cap \tilde{\tau}$. Here we assume that  we merge two adjacent rectangles to a bigger rectangle whenever possible, so that the horizontal sides of these rectangles are $\tilde{\tau}$ and \textit{distinct} lifts of $\tau$, and we also assume that the vertical sides are geodesic segments of leaves of $\tilde{\lambda}$.\\

These rectangles descend to the collection of rectangles $R_1,R_2,\ldots R_n$ on the surface $X$ whose union contains the lamination $\lambda$. Their horizontal sides determine intervals $I_1,\ldots I_n$ on the arc $\tau$. Since $\lambda$ is arational, it has  complementary regions $T_1,T_2,...T_m$  which are truncated ideal hyperbolic triangles (here $n$ and $m$ depend only on the genus).\\

We define
\begin{equation}
\mathcal{T}_\epsilon = R_1\cup R_2\cup \cdots R_n
\end{equation}

Recall here that $\epsilon$ determines the choice of length of $\tau$ above - as shall be specified in  Lemma \ref{lem:thinrect}.\\

Along the grafting ray the rectangles get wider, and one gets a decomposition of each grafted surface $X_t$ with the same combinatorics of the gluing.

\begin{figure}
  \centering
  \includegraphics[scale=0.6]{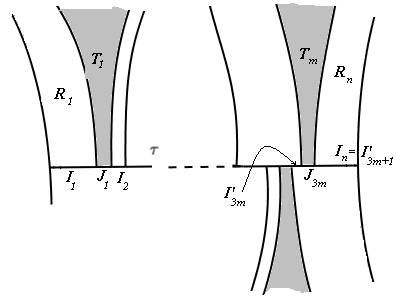}\\
  \caption{The train track decomposition and the labelling of the sub-arcs of $\tau$.}
\end{figure}

\subsubsection{Labelling sub-arcs}

For later use, we denote the sub-arcs of the arc $\tau$ that are the horocyclic edges of $T_1,\ldots T_m$ by $J_1,\ldots J_{3m}$ (these in fact belong to the collection $\{J_i^\pm\}$) and the remaining sub-arcs in $\tau\setminus \{J_1\cup\cdots J_{3m}\}$ by $I^\prime_1,\ldots I^\prime _{3m+1}$. Note that for each $1\leq i\leq n$ we have
\begin{equation}\label{eq:finer}
I_i = \bigcup\limits_{k\in S_i} I^\prime_k
\end{equation}
where $S_i$ is a finite subset of $\{1,2,\ldots 3m+1\}$, and $S_i\cap S_j = \emptyset$ for $i\neq j$.

\subsubsection{Dimensions}
The \textit{total width} of a rectangle $R_i$ shall be the supremum of the lengths in the Thurston metric  (see Definition \ref{defn:projmet}) of the segments $l\cap R_i$, where $l$ is a leaf of the transverse foliation $\mathcal{F}$.\\

The \textit{hyperbolic width} of a rectangle $R_i$ is defined to be the supremum of the hyperbolic lengths of the leaves of $\mathcal{F}$ that intersect $R_i$. The \textit{euclidean width} of a rectangle $R_i$ is defined to be its total width \textit{minus} its hyperbolic width.\\ 

The \textit{height} of the rectangle is the hyperbolic length of one of the vertical geodesic sides.

\begin{lem}[Long, thin train-track]\label{lem:thinrect}
If  the hyperbolic length of $\tau$ is sufficiently small, the height of each of the rectangles $R_1,R_2,\ldots R_n$ is greater than $\frac{1}{\epsilon^2}$, and the hyperbolic width is less than $\epsilon$.
\end{lem}
\begin{proof}
Let the hyperbolic length of the arc $\tau$ be $L$. Then the two horizontal sides of a rectangle being intervals on the segment $\tau$ have length less than $L$. It follows from elementary hyperbolic geometry that on the (ungrafted) surface $X$ the two vertical geodesic sides remain within $L$ of each other, and each piecewise-horocyclic leaf of $\mathcal{F}$ between them has length at most $O(L)$. It is a fact that if the horocyclic edges of a truncated ideal triangle have length $O(L)$, then the height at which the triangle is truncated is at least $\ln\frac{1}{L}$. This is the height of the rectangle. Clearly $\ln\frac{1}{L}\to \infty$ as $L\to 0$, so we can choose $L<\epsilon$ small enough so that the statement of the lemma holds.
\end{proof}

Henceforth, we shall assume that $\tau$ was chosen short enough such that the conclusions of the above lemma hold.\\

Each rectangle has a euclidean width equal to the transverse measure $\mu(t\lambda \cap R)$ which goes to infinity as the grafting time $t\to \infty$. Hence we also have:

\begin{lem}[Total width]\label{lem:widerect}
There is a $T>0$ such that for all $t>T$, the total width of each rectangle in the above decomposition is greater than $1/\epsilon^2$.
\end{lem}


\subsubsection*{The decomposition $\mathcal{D}$}
It will be useful to organize the above collection of rectangles and truncated ideal triangles into the following decomposition of a sufficiently grafted surface into rectangles and pentagons:\newline

Namely, first decompose each ideal triangle $T_j$ into three pentagons by including geodesic edges from its centroid $p_j$ to the midpoints of the horocyclic sides. Each pentagon is thus a `$2\pi/3$-sector' of a truncated ideal hyperbolic triangle.\\

Each $T_j$  is adjacent to rectangles from the collection $\{R_i\}$ on its three geodesic sides. By Lemma \ref{lem:widerect}, for a sufficiently-grafted surface there is a grafted portion of euclidean width much greater than $4$ adjacent to each geodesic side. We thicken the pentagons by  trimming subrectangles of euclidean width $2$ from the adjacent rectangles and appended to the truncated sectors of $T_j$ (obtained above). \newline

\begin{figure}[h]
  \centering
  \includegraphics[scale=0.5]{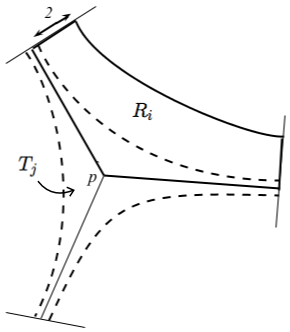}\\
  \caption{In the decomposition $\mathcal{D}$, each  truncated ideal triangle $T_j$ is divided into three sectors which are thickened to form pentagons, by appending a portion of the rectangle adjacent to the geodesic sides of $T_j$ (shown dotted).}
\end{figure}

This trimming-and-appending results in the euclidean (and total) widths of the rectangles $\{R_i\}$ decreasing by $4$. By abuse of notation, we continue to denote these trimmed rectangles as  $\{ R_i\}_{1\leq i\leq n}$, and we denote the pentagons (thickened sectors) as $\{ P_j\}_{1\leq j\leq 3m}$. These form the pieces of this new decomposition, which we shall refer to as the decomposition $\mathcal{D}$.

\subsection{A compendium of quasiconformal maps}

There are two types of pieces in the decomposition $\mathcal{D}$  of the grafted surface described in the previous section, rectangles $R_i$, and pentagons  $P_j$, which can be further decomposed into truncated $2\pi/3$-sectors and width-$2$ rectangles. We shall eventually construct some quasiconformal maps of these pieces to the euclidean plane which we shall glue to form a quasiconformal map of the grafted surface to the singular flat surface.\newline

We first isolate as lemmas a few quasiconformal maps that will be useful at several steps of the actual construction. Since we would need some control on the quasiconformal dilatation of the final map, we take care to ascertain the distortion at all points of the domain.\newline

We use repeatedly  the following facts about quasiconformal maps (see, for example, \cite{Ah1}):\newline

A. If the partial derivatives of a $C^1$ map $f$  between planar domains  satisfy\\

\begin{center}
$\left\vert \frac{\lVert f_x\rVert}{\lVert f_y\rVert} -1 \right\vert< \epsilon$ ,    and    $\lvert \langle f_x,f_y\rangle \rvert < \epsilon$\newline
\end{center}
at a point (where recall $\epsilon$, as throughout this paper, is sufficiently small), its quasiconformal dilatation there is $1 + C\epsilon$ for some universal $C>0$.\\

B. If a map $f$ is a homeomorphism onto its image, and it is quasiconformal on the domain except for a measure zero set (typically a collection of $C^1$ arcs), then $f$ is quasiconformal everwhere on the domain.\newline

C. If a homeomorphism $f$ is $K$-quasiconformal, so is its inverse $f^{-1}$.

\subsubsection*{Straightening map}

\begin{defn}\label{defn:avert}
An arc on the euclidean plane to be $\epsilon$-\textit{almost-vertical} if it is a portion of a graph $x=g(y)$ where $g$ is a $C^1$-function and $\lvert g'(y)\rvert <\epsilon$.
\end{defn}

The following lemma will be used in various contexts to map a rectangle with ``almost-vertical" sides to an actual rectangle by an almost-conformal map that is also height-preserving.

\begin{lem}\label{lem:qcstraight}
Let $R$ be a planar region that is bounded by two sides that are parallel horizontal line segments and two arcs which are the graphs $x = g_1(y)$ and $x=g_2(y)$ over the interval $0\leq y\leq a$ on the $y$-axis, where $g_1$, $g_2$ are $C^1$-functions such that $g_2(y)>g_1(y)>0$ for all $0\leq y\leq a$. Then there is a height-preserving quasiconformal map $f$ from $R$ to a euclidean rectangle of height $a$ and width $b$, where $b = \sup\limits_{0\leq y\leq a} \left(g_2(y)-g_1(y)\right)$.\\
Moreover, if at any point $p\in R$ at height $y$ we have:\\
(i) $\left\vert g_1^\prime(y)\right\vert <\epsilon$,\\
(ii) $\left\vert g_2^\prime(y)\right\vert <\epsilon$, and \\
(iii) $\left\vert \frac{b}{g_2(y) - g_1(y)} - 1\right\vert <\epsilon$, \\
then the quasiconformal distortion of $f$ at $p$ is $(1 + C\epsilon)$, for some universal constant $C>0$.
\end{lem}
\begin{proof}
The map $f$ one constructs is one that ``stretches"  horizontally the right amount at each height:
\begin{equation*}
(x,y)\mapsto \left( b\frac{x - g_1(y)}{g_2(y) - g_1(y)}, y\right)
\end{equation*}

Computing the partial derivatives of $f$ we get:
\begin{equation*}
 f_x  = \left\langle \frac{b}{g_2(y) - g_1(y)}, 0\right\rangle
\end{equation*}
\begin{equation*}
 f_y = \left\langle   -A(y)g_1^\prime(y) - A(y)^2(g_2^\prime(y)-g_1^\prime(y))\left( \frac{x-g_1(y)}{b}\right),1\right\rangle
\end{equation*}
where $A(y) =  \frac{b}{g_2(y) - g_1(y)}$.\\

At a height $y$ where the estimates (i)-(iii) hold, using them and the observation that $\frac{x-g_1(y)}{b} \leq 1$, we get:
\begin{center}
$\left\vert \frac{\lVert f_x\rVert}{\lVert f_y\rVert} -1 \right\vert <C\epsilon$ ,    and    $\lvert \langle f_x,f_y\rangle \rvert <C \epsilon$.\newline
\end{center}
for some universal constant $C>0$, and the statement on quasiconformal distortion follows from property A above.
\end{proof}
 
\subsubsection*{Maps straightening a horizontal foliation}

\begin{figure}
  \centering
  \includegraphics[scale=0.4]{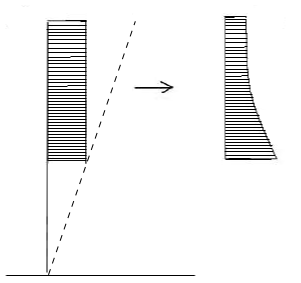}\\
  \caption{The map to $\mathbb{R}^2$ that straightens the horocyclic foliation of an $\epsilon$-thin hyperbolic region (shown on the left in the upper half plane) is almost-conformal (Lemma \ref{lem:qcthin}).}\label{fig:img}
\end{figure}

 \begin{lem}\label{lem:qcthin}
Let $R$ be a rectangular region in $\mathbb{H}^2$ in the upper half-plane model bounded by two (vertical) geodesic sides and two (horizontal) horocyclic sides such that the hyperbolic width is $\epsilon$ (sufficiently small). Note that $R$ is foliated by horocyclic segments.\newline
Then the map $f:R\to \mathbb{R}^2$ that \newline
(i) takes the left edge to a vertical segment\newline
(ii) maps each horocyclic leaf to a horizontal line, and \newline
(iii) is distance-preserving along each horocyclic leaf,\newline
is $(1+C\epsilon)$-quasiconformal for some universal constant $C>0$.\newline
Moreover, the right edge of $R$ is mapped to an almost-vertical segment.
\end{lem}
\begin{proof}
Let  $R$ to be a rectangle in the upper-half plane model of $\mathbb{H}^2$ as described, lying above the line $y=1$ and the left side lying on the $y$-axis. The map $f$ is 
\begin{equation}
(x,y) \mapsto\left (\frac{x}{y},\ln y\right)
\end{equation}
Note that the hyperbolic width of $R$ being at most $\epsilon$ implies that  
\begin{equation}\label{L1}
\left\lvert\frac{x}{y}\right\rvert = O(\epsilon)
\end{equation}
for all points in $R$.\newline
One can compute the derivatives of the above map to get the dilatation:
\begin{equation}
\frac{\lVert f_x\rVert}{\lVert f_y\rVert} = \frac{\frac{1}{y}}{\sqrt{\frac{x^2}{y^4} + \frac{1}{y^2}}} = \frac{1}{\sqrt{\frac{x^2}{y^2} + 1}} =1 + O(\epsilon)
\end{equation}
and
\begin{equation}
\lvert \langle f_x,f_y\rangle \rvert = \frac{x}{y^3} = O(\epsilon)
\end{equation}
by \eqref{L1} and the fact that $y\geq 1$.\newline

This computation also verifies that the image of the right edge is a graph over the $y$-axis of small derivative.
\end{proof}
 
A similar proof yields the following:

\begin{lem}[Inside-out version]\label{lem:thinflipped}
Let $R$ be the region as in the previous lemma. Assume that the height of $R$ is greater than $1$. Then the map $f:R\to \mathbb{R}^2$ that takes the \textit{right} edge to a vertical segment, and satisfies (ii) and (iii), is almost-conformal.
\end{lem}

The following observation involves euclidean regions one gets by grafting (see Figure 9 in \S4.3).

\begin{lem}\label{lem:qceuc}
Let $R$ be the region on the upper half-plane consisting of all $z\in \mathbb{C}$ satisfying $a \leq \lvert z\rvert \leq b$, and $\alpha \leq arg(z)\leq \pi/2$, for some $0<\alpha <\pi/2$. We equip $R$ with the metric $dz/\lvert z \rvert$. This region $R$ is foliated by arcs $F_l =\{ le^{i\theta}\vert  \alpha \leq \theta \leq \pi/2\}$ for $a\leq l\leq b$. Then there is an isometry $g$ from $R$ to a rectangle in $\mathbb{R}^2$ of height $\ln{\frac{b}{a}}$ and width $\pi/2-\alpha$, such that it maps each horizontal leaf  of $F$ to a horizontal line in a length-preserving way. 
\end{lem}
\begin{proof}
It can be checked that the conformal map  $z \mapsto \pi/2 + i\ln z$ is the required isometry.
\end{proof}

\subsubsection*{Almost-isometries and quasiconformal extensions}

\begin{defn}\label{defn:aisom}
Let $L$ and $L^\prime$ be two intervals with a metric (eg., two line segments on the plane). Then a homeomorphism $f:L\to L^\prime$ is said to be an \textit{$\epsilon$-almost-isometry} if\newline
(1) $f$ is $C^1$ with dilatation $d$ (the derivative of $f$ when $L$ is parametrized by arclength) satisfies $\lvert d-1\rvert <\epsilon$.\newline
(2) the lengths of any sub-interval of $L$ and its image in $L^\prime$ differ by an \textit{additive} error less than $\epsilon$.
\end{defn}
\textit{Remark.} For brevity, we shall often use `$\epsilon$-almost-isometric' or just `almost-isometric' to mean `$M\epsilon$-almost-isometry' for some (universal) constant $M>0$. \newline

Here are some immediate observations, whose proofs we omit:

\begin{lem}\label{lem:noskew0}
If $f:L\to L^\prime$ and $g:L^\prime \to L^{\prime\prime}$ are $\epsilon$-almost-isometries, then $f^{-1}$ and $g\circ f$ are $2\epsilon$-almost-isometries.
\end{lem}

\begin{lem}\label{lem:noskew1}
If  the difference of the lengths of the segments $\lvert l(L)-l(L^\prime)\rvert < \epsilon$ then the (orientation-preserving) affine map $f:L\to L^\prime$ is an $\epsilon$-almost-isometry.
\end{lem}

\begin{lem}\label{lem:noskew2}
If $L$ is subdivided into subintervals $A_1,A_2,\ldots A_N$ and $L^\prime$ into subintervals $A^\prime_1,A^\prime_2,\ldots A^\prime_N$ and the restrictions $f_{\vert A_i}:A_i\to A^\prime_i$ are $\epsilon$-almost-isometries.
Then $f:L\to L^\prime$ is an $N\epsilon$-almost-isometry.
\end{lem}

The following lemma is how the condition would be useful for in our constructions.

\begin{lem}\label{lem:qcnoskew}
Let $R$ and $R^\prime$ be two planar rectangles of the same height $h>1$ and moduli $m,m^\prime$ greater than $1$. Suppose $f:\partial R\to \partial R^\prime$ is vertex-preserving homeomorphism that maps the left and right edges by an isometry and is a $\epsilon$-almost-isometry on the top and bottom edges. If $\epsilon$ is sufficiently small, then $f$ can be extended to a $(1+C\epsilon)$-quasiconformal map from $R$ to $R^\prime$,where $C>0$ is some universal constant.
\end{lem}
\begin{proof}
Let $S$ be a Euclidean rectangle of modulus $m$. Then by mapping the rectangle to a unit disk and applying the Ahlfors-Beurling extension (\cite{AB}), any vertex-preserving piecewise-affine $C^1$ homeomorphism $f:\partial S \to \partial S$ of dilatation of the order of $1 + C\epsilon$ can be extended to a homeomorphism $f: S \to S$ which is $(1 + K\epsilon)$-quasiconformal, where $K$ depends only on $C$ and the modulus $m$ of $R$..  $K$ gets worse (larger) for the modulus $m$ very large or very small, but if $m$ lies in a compact set, say $[1,2]$, we get a uniform upper bound for $K$.\newline

The strategy is to subdivide $R$ into smaller rectangles of moduli between $1$ and $2$, and use the above fact.\newline

\begin{figure}
  \centering
  \includegraphics[scale=0.6]{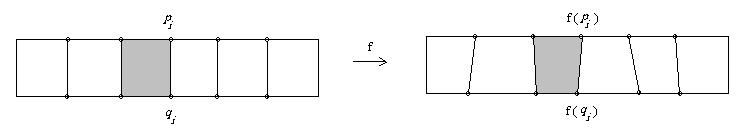}\\
  \caption{A boundary map that is ``almost-isometric" can be extended to an almost conformal map, by subdividing into smaller rectangles (Lemma \ref{lem:qcnoskew}). }\label{fig:img}
\end{figure}
Let $p$ and $q$ be the top and bottom corners on the left side. Choose a collection of points $p_1,p_2,\ldots,p_n$ on the top side, and $q_1,q_2,\ldots,q_n$  on the bottom such that\newline
(i) $l(\bar{pp_i})=\l(\bar{qq_i})$ for each $1\leq i\leq n$.\newline
(ii) On each side the points subdivide it into subintervals having lengths between $h$ and $2h$. (This is possible because the modulus $m>1$.)\newline

Consider the rectangles $R_1,R_2,\ldots,R_n$ obtained by connecting each pair $p_i,q_i$ by a straight line. By (ii) above, the modulus of each $R_i$ is between $1$ and $2$. \newline

Consider the images $f(p_i)$ and $f(q_i)$ on the top and bottom sides of $\partial R^\prime$.
By (i) above, property (2) of the definition of almost-isometry, and the height $h>1$, we have that the straight line joining $f(p_i)$ and $f(q_i)$ is $\epsilon$-almost-vertical, for each $i$. We call the resulting collection of almost-rectangles $R^\prime_1,R^\prime_2,\ldots, R^\prime_n$. We extend the boundary map $f$ to map each line from $p_i$ to $q_i$ to the line from $f(p_i)$ and $f(q_i)$ by an affine stretch (of dilatation $1 + O(\epsilon)$).\newline

By Lemma \ref{lem:qcstraight} and a horizontal affine scaling we have an almost-conformal map $h$ from each $R_i$ to $R^\prime_i$. To correct for this map differing from the map $f$ on $\partial R_i$, we consider the map $f\circ h^{-1}\vert_{\partial R^\prime_i}:\partial R^\prime_i \to \partial R^\prime_i$. This has dilatation $1 + O(\epsilon)$ and can be extended to an almost-conformal map $g$ by the Ahlfors-Beurling extension (the moduli of $R^\prime_i$ also lie in a compact subset slightly larger than $[1,2]$).  The map $g\circ h:R_i\to R^\prime_i$ agrees with $f$ on $\partial R_i$.\newline

These almost-conformal maps of each $R_i$ to $R^\prime_i$ piece together to give an almost-conformal map of $R$ to $R^\prime$. (The property of almost-conformality extends across the intermediate arcs.) \end{proof}

\textit{Remark.} Such an almost-conformal extension may fail to exist for a boundary map that is merely $C^1$ with small dilatation, without the additional condition (2) of Definition \ref{defn:aisom}. 

\subsection{Map for a rectangular piece}

Consider a typical rectangle $R=R_i$  in our decomposition of the grafted surface $X_t$ .  $R$ is bounded by geodesics on each vertical side and by leaves of the transverse horocyclic foliation on each horizontal side. This horocyclic foliation $\mathcal{F}\cap R$ gives a $C^1$ foliation of the rectangle.  Geodesic arcs belonging to the lamination $\lambda$ cut across the rectangle transverse to the foliation, and $R \setminus \lambda$ has countably many hyperbolic components, bounded by horizontal horocyclic arcs and vertical geodesic arcs. The goal of the section is to construct a quasiconformally equivalent `euclidean' model for $R$.

\subsubsection*{Working in the universal cover}

We shall work in the universal cover $\widetilde{X_t}$ of $X_t$, where we consider a (fixed) lift $\tilde{R}$ of $R$. Moreover we shall assume that the developing map $dev:\widetilde{X_t} \to \mathbb{C}P^1$ of the complex projective structure on $X_t$ is injective (and a homeomorphism) on $\tilde{R}$. It is injective whenever the transverse measure across $R$ is sufficienty small, so this condition can be ensured by subdividing $R$ vertically. The map for $R$ having arbitrary transverse width can then be obtained by piecing these divisions together: properties of quasiconformal extension tell us that if the map for each piece is almost-conformal, so is the concatenated map.\newline

By abuse of notation, we shall identify $\tilde{R}$ with its homeomorphic image on $\mathbb{C}P^1$, and consider it a planar domain (since it is a proper subset of $\mathbb{C}P^1$ it lies in an affine chart). \newline

The horizontal foliation $\mathcal{F}_{\vert R}$ lifts to the universal cover and to $\tilde{R}$ via the developing map. We denote it by $\widetilde{\mathcal{F}}$. The Thurston metric on $R\subset X_t$ is locally isometric to the projective metric on $\tilde{R}$ via $dev\circ u^{-1}$ where $u:\widetilde{X_t}\to X_t$ is the universal covering  (see Definition \ref{defn:projmet}).

\subsubsection*{A finite approximation}

$\widetilde{X_t}$ (identified as a domain of $\mathbb{C} \subset \hat{\mathbb{C}}$) is thought of as obtained by grafting the upper half plane identified as the universal cover of $X$, along the lifted measured lamination $\tilde{\lambda}$.  The grafting locus consists of a collection of infinitely many geodesics that can be approximated by a sequence of  \textit{finite} weighted subsets that produce approximations $\widetilde{X_i}$. (This can be thought of as approximating the Borel measure induced by $\lambda$ on $S^1\times S^1 \setminus \Delta$ by a sequence of sums of Dirac measures). We can further assume that these finite approximations\newline
(i) always include the geodesics $\gamma_l,\gamma_r$ that form the left and right edge of $\tilde{R}$ on $\tilde{X}$, and\newline
(ii) is maximal in the sense that the complement consists of ideal triangles.\newline

By (ii), there is a piecewise-horocyclic foliation $\widetilde{\mathcal{F}_i}$ on each finite approximation $\widetilde{X_i}$.\newline

\textit{Notation.} In what follows,  given a sub-arc $s$ of a leaf of $\mathcal{F}$, we shall denote  its \textit{hyperbolic} length as $l_h(s)$,  its \textit{total} length in the Thurston metric as $l(s)$, and  its \textit{euclidean} length as $l_e(s)$, which is defined as the difference $l(s) -l_h(s)$.\\

Let $s \subset \gamma_l$ denote the left edge of $\tilde{R}$. Then by (i) we can define a rectangle $\widetilde{R_i}$ on $\widetilde{X_i}$ as having $s$ as the left edge, leaves of the foliation $\widetilde{\mathcal{F}_i}$ as the two horizontal edges, and a segment on $\gamma_r$ as the right edge. Note that for all $i$, we can define the leaf  $l_i^y \in \widetilde{\mathcal{F}_i}\cap \widetilde{R_i}$ at `height $y$' to be the leaf intersecting $s$ at a distance of $y$ from the lower endpoint of $s$. We also define $l^y$ to be the leaf of $\widetilde{\mathcal{F}}\cap \tilde{R}$ at height $y$. \\

\textit{Remark.} Since the holonomy along $\widetilde{\mathcal{F}_i}$  (and $\widetilde{\mathcal{F}}$) preserves the hyperbolic length along leaves of $\widetilde{X_i}$,  this definition of `height' is well-defined,  that is, the leaf $l_i^y$ (and $l^y$) are at height $y$ all along  $\widetilde{R_i}$ (and $\widetilde{R}$).

\begin{defn}
A \textit{conformal metric} on a Riemann surface $\Sigma$ is a metric given by $\rho(z) \lvert dz\rvert^2$ in each local coordinate chart, for some function $\rho$  (the \textit{conformal factor}) on $\Sigma$. It is said to be of class $C^{1,1}$ if $\rho$ is differentiable with Lipschitz derivatives, in which case its $C^{1,1}$-\textit{norm} is defined to be $\lVert \rho\rVert_{1,1}$. A family of conformal metrics are said to \textit{converge pointwise} if the corresponding conformal factors converge pointwise.
\end{defn}

The following lemma states the known convergence results that have also been mentioned in \S 2 while describing grafting for general laminations.

\begin{lem}\label{lem:approx}
For the above sequence of finite approximations $\widetilde{X_i}$ the following are true:\newline
(i) The Thurston (or projective) metric on $\widetilde{X_i}$ is a conformal metric of class $C^{1,1}$.  They converge pointwise to the Thurston metric on $\widetilde{X_t}$. Moreover, the $C^{1,1}$ norms remain bounded as $i\to \infty$.\newline
(ii) The horocyclic foliations $\widetilde{\mathcal{F}_i}\to \widetilde{\mathcal{F}}$ in the sense that for all $0\leq y \leq l(s)$ we have $l_i^y \to l^y$ in the Hausdorff metric on compact subsets of $\mathbb{C}$. \newline
(iii) The rectangles $\tilde{R_i} \to \tilde{R}$ in the Hausdorff metric on compact subsets of $\mathbb{C}$.
\end{lem}
\begin{proof}
Part (i) involving the regularity of the Thurston metrics is a standard result (see \cite{KulPink} or Lemma 2.3.1 in \cite{ScanWolf}). Part (ii) follows from part (i) and the `bending' description of grafting (see \S2): in our choice of approximates the corresponding locally convex pleated planes corresponding to $\widetilde{F_i}$ converge in the Gromov-Hausdorff sense to the pleated plane corresponding to $\widetilde{F}$ by work of Epstein-Marden (\cite{EpMar}) and Bonahon(\cite{Bon}), and the developing maps (which are the hyperbolic Gauss maps from these pleated planes) converge uniformly on compact sets. This implies that as $i\to \infty$ the leaf segments $l_i^y$ on $\tilde{R_i}$ converge to \textit{some} horizontal leaf segment of $\tilde{R}$ in the Hausdorff metric, and part (i) implies that the height of the latter is also $y$.\\ Part (iii) is now an immediate consequence, since the rectangles  $\tilde{R_i}$ are defined in terms of the vertical left edge $\gamma_l$ and the segments $l_i^y$ for $0\leq y \leq l(s)$, which converge to those of $\tilde{R}$.
\end{proof}

\subsubsection*{Map for the finite approximation}

\begin{lem}\label{lem:finitemodel}
If the hyperbolic width of $\widetilde{R_i}$ is $\epsilon$ (sufficiently small) then there exists a $(1 + C\epsilon)$-quasiconformal map $\widetilde{f_i}$ from $\widetilde{R_i}$ to the euclidean plane, for some  universal constant $C>0$, that satisfies the following:\newline
(i) It takes the lower endpoint of the left vertical geodesic side $s$ to the origin.\newline
(ii) It is an isometry of $s$ onto a segment on the $y$-axis.\newline
(iii) Each leaf of $\widetilde{\mathcal{F}}\cap \widetilde{R_i} $  is mapped to a horizontal line in a length-preserving way.
\end{lem}

\begin{figure}\label{fig:graft1}
  \centering
  \includegraphics[scale=0.52]{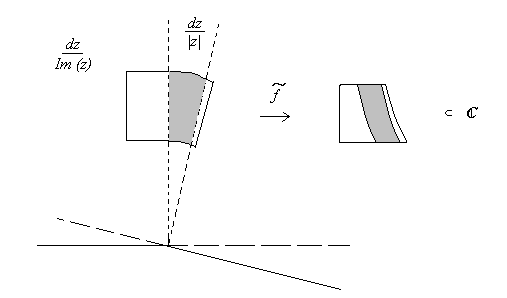}\\
  \caption{The map for the finite approximation case. The figure on the left shows a rectangle $\widetilde{R_i}$: the hyperbolic part (unshaded) is mapped by the straightening map of Lemma \ref{lem:qcthin}, and the euclidean part (shown shaded) is then spliced in by the map from Lemma \ref{lem:qceuc}.}
  \end{figure}

\begin{proof}
The map $\widetilde{f_i}$ is uniquely determined and injective by conditions (i)-(iii), and is also $C^1$ since the the grafted metrics are $C^1$ and the foliation is $C^1$, the horizontal leaves being integral curves of a nowhere-zero Lipschitz vector field. In particular it is a quasiconformal map, and it remains to show that the dilatation is $1 + O(\epsilon)$, and it is enough to check that almost everywhere.\newline
  
Recall that the projective metric (Definition \ref{defn:projmet}) is the Poincar\'{e} metric on the maximal disk at every point. If one starts with a maximal disk in the upper half plane model with metric $dz/Im(z)$,  and the $x=0$ axis is the lift of the grafting curve, then grafting introduces a sector with the (euclidean) metric $dz/\lvert z\rvert$.\newline

$\widetilde{R_i}$ consists of regions that alternately lie in the hyperbolic part (of width $O(\epsilon)$) and the euclidean part (see the figure) with finitely many separating geodesic arcs.\newline

Collapse the euclidean regions of $\widetilde{R_i}$ to get a rectangle $\widetilde{R^\prime}$ of hyperbolic width $O(\epsilon)$. By Lemma \ref{lem:qcthin} there is a map $f_0$ of $\widetilde{R^\prime}$ to $\mathbb{R}^2$ satisfying (i)-(iii) above. Note that the proof of Lemma \ref{lem:qcthin} shows that separating geodesic arcs are mapped to almost-vertical arcs on the plane.\newline

Starting with $f_0$ we now inductively splice in each euclidean region to $\widetilde{R^\prime}$ and extend the map already constructed, to the larger domain, such that (iii) is satisfied. ((i) and (ii) are automatically satisfied for all these extensions.) By the above observation on the interface arcs being almost-vertical, and Lemma \ref{lem:qceuc}, these extensions are almost-conformal on each region.\newline

Since the interface arcs are of measure zero, the final map $f_n = \widetilde{f_i}$ thus constructed is $(1 + O(\epsilon))$-quasiconformal almost everywhere, as required.\newline
\end{proof}

\subsubsection*{Taking a limit}
By the above lemma we now have a sequence 
\begin{equation*}
\widetilde{f_i}:\widetilde{R_i}\to \mathbb{R}^2
\end{equation*}
of almost-conformal maps.

\begin{lem}\label{lem:model1}
The maps $\widetilde{f_i}$  converge uniformly to an almost-conformal map $\tilde{f}:\tilde{R}\to \mathbb{R}^2$, that satisfies the conditions (i)-(iii) in the above lemma.
\end{lem}
\begin{proof}
The uniform convergence follows from parts (i) and (ii) of Lemma \ref{lem:approx}: part (ii) says the each leaf of $\widetilde{R_i}$ converges \textit{as a set} to the corresponding leaf of $\tilde{R}$, and by part (i) the \textit{lengths} also converge, and then the uniform convergence follows from the definition of $f$ (distance preserving along the leaves). To get the statement on almost-conformality we employ a trick of considering the sequence of \textit{inverse} maps $\widetilde{g_i} = \widetilde{f_i}^{-1}$. For an arbitrary $x\in \tilde{R}\setminus \partial\tilde{R} $ there is, by part (iii) of Lemma \ref{lem:approx}, an open neighborhood $\tilde{U}\subset \mathbb{R}^2$ containing $\tilde{f}(x)$, such that $\tilde{U}$ is contained in $\widetilde{f_i}(\widetilde{R_i})$ for sufficiently large $i$. The sequence $\widetilde{g_i}\vert_{\tilde{U}}$ are a uniformly converging sequence of $(1+O(\epsilon))$-quasiconformal maps from a fixed domain $\tilde{U}$ to $\mathbb{C}$. The limit $\tilde{g}=\tilde{f}^{-1}$ is hence $(1+O(\epsilon))$-quasiconformal on $\tilde{U}$, and so is $\tilde{f} = \tilde{g}^{-1}$ in a neighborhood of $x$.\newline
Note that $f$ is height-preserving: this follows from parts (ii) and (iii) of Lemma \ref{lem:finitemodel} (which are preserved in the limit).
\end{proof}

\subsubsection*{The almost-conformal model}
By the previous lemma we have obtained an almost-conformal map $\tilde{f}$ from $\tilde{R}$ to $\mathbb{R}^2$. Together with the local isometry $dev\circ p^{-1}$ this gives an almost-conformal map of $R\subset X_t$ to a planar domain. We conclude the construction of a quasiconformal model for $R$ by noting that this planar domain can be almost-conformally straightened to a rectangle.\newline

\textit{Notation.} Recall that the total width $W$ of the rectangle $R$ is the supremum of the lengths of the leaves of $\mathcal{F}\cap R$, and the euclidean width $W_e$ of $R$ is the supremum of the \textit{euclidean}  lengths of the leaves of $\mathcal{F}\cap R$.\\

By construction (see Lemma \ref{lem:thinrect}), the hyperbolic width of $R$ is less than $\epsilon$, and its height $h$ is greater than $1$. We also assume that its total width $W$ is greater than $1$, that is, one has grafted enough for Lemma \ref{lem:widerect} to hold.

\begin{lem}[Map for a rectangle]\label{lem:rectModel}
With the above assumptions, there exists a  height-preserving $(1+C\epsilon)$-quasiconformal map $\bar{f}$ from $R$ to a euclidean rectangle of width $W_e$ and height $h$. (Here $C>0$ is some universal constant.)
\end{lem}
\begin{proof}
By the previous lemma we have a height-preserving almost-conformal map $f = \tilde{f}\circ dev\circ p^{-1}$ from $R$ to a planar region $D$. The euclidean width $W_{e}(y)$ (total minus hyperbolic) of a leaf at height $y$ is a constant independent of height (it only depends on the total transverse measure of $R$). Since the hyperbolic width $W_h(y)$ at height $y$ is $O(\epsilon)$ (Lemma \ref{lem:thinrect}), we have that $\frac{W}{W(y)} = 1 + O(\epsilon)$ where $W(y) = W_h(y) + W_e(y)$ is the total width of $D$ (and also of $R$) at height $y$. Since the two vertical sides of $R$ are geodesic segments of length $h>1$ on the (ungrafted) hyperbolic surface at distance $O(\epsilon)$, it follows from hyperbolic geometry that $\left\lvert \frac{dW_h(y)}{dy} \right\rvert = O(\epsilon)$. Hence by an application of Lemma \ref{lem:qcstraight} we obtain a height-preserving almost-conformal map from $D$ to a euclidean rectangle of width $W_e$, and the required map $\bar{f}$ from $R$ to $W_e$ is obtained by precomposing this with $f$.
\end{proof}

\begin{cor}[Almost-isometric]\label{cor:rectnoskew}
Let $L$ and $\bar{L}$ be the top and bottom edges of $R$. The map $\bar{f}$ constructed above is an $\epsilon$-almost-isometry on $L$ and $\bar{L}$, and isometric on the other two sides. 
\end{cor}
\begin{proof}
Consider the top edge $L$. The map of Lemma \ref{lem:model1} is an isometry of $L$ and the map (from Lemma \ref{lem:qcstraight}) used in straightening step in the proof of the previous lemma is affine on horizontal lines, hence the composition (the map $\bar{f}$) is affine on $L$. Since $l(L) = W(h)$ and $l(\bar{f}(L))=W_e$ differ by $O(\epsilon)$ we can apply Lemma \ref{lem:noskew1} to conclude that $\bar{f}$ is an $\epsilon$-almost-isometry on $L$.
The proof for the bottom edge $\bar{L}$ is identical.\\
The property of being ``height-preserving" implies that $\bar{f}$ is isometric on the vertical (left and right) sides of $R$.
\end{proof}

\subsection{Map for a pentagonal piece}
The purpose of this section is to define a quasiconformal map of each pentagonal piece $P$ of the decomposition $\mathcal{D}$ (\S4.1.3) to a planar region that is almost-conformal for ``most" of $P$ (Lemma \ref{lem:hexmap}). The map of is obtained by ``straightening" leaves of a foliation through $P$. These maps together with those from the previous section, will be assembled in \S4.5 to define the map of the grafted surface.

\subsubsection*{The foliation $\mathcal{F}$} An ideal hyperbolic triangle has a partial foliation by horocyclic arcs which restricts to give a partial foliation of the ``$2\pi/3$-sector" $\hat{S}$. When realized in the upper half plane $\mathbb{H}^2$ with $\gamma$ a vertical geodesic and $p$ the point $\frac{i\sqrt{3}}{2}$, the leaves of \textit{non-negative height} are the horizontal segments starting at $y=1$. In general, the \textit{height} of a leaf is the logarithm of its $y$-coordinate of the point where it intersects $\gamma$. We shall work with a horizontal foliation (that we continue to denote by $\mathcal{F}$) that extends the horocylic foliation to the whole of  $\hat{S}$ as follows:  \newline

Let $a$ be the geodesic arc from $p$ which is orthogonal to $\gamma$. This divides $\hat{S}$ into two parts, and on one of them we define the modified $C^1$ foliation to be one that \\
(1) agrees with the horocyclic foliation for height greater than $D= \ln(1/\epsilon)$ (when the leaves have width less than $\epsilon$),\\
(2) interpolates between the leaf at height $D$ and the arc $a$ in such a way that the lengths of the leaves is a decreasing $C^1$ function of (nonnegative) height, and\\
(3) has each leaf orthogonal to $\gamma$.\\

The foliation on the other part of $\hat{S}$ is obtained by reflecting across $a$. The above length function is $C^1$ except at height $0$.

\begin{figure}
  \centering
  \includegraphics[scale=0.3]{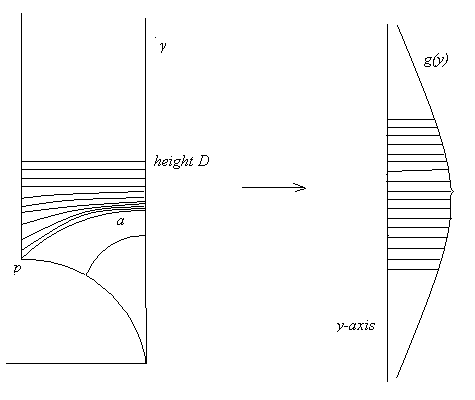}\\
  \caption{The map for the sector (shown in the upper half space model of $\mathbb{H}^2$) in Lemma \ref{lem:trunc} straightens the interpolating foliation between $a$ and the horocyclic leaf at height $D$. }\label{fig:img}
\end{figure}

\subsubsection*{Constructing the map} Let $S_H$ denote a ``truncated $2\pi/3$-sector" truncated at leaves of $\hat{\mathcal{F}}$ at height $H$ (we take $H> D =  \ln(1/\epsilon)$).
Recall from \S 4.1.3 that a pentagonal piece $P$ is made by appending a strip $S$ of the adjacent grafted rectangles of (euclidean) width $2$ to the geodesic side of such an $S_H$ coming one of the ideal hyperbolic triangles from the collection $\{T_j\}_{1\leq j\leq m}$.\\

That is, we have:
\begin{center}
$P= S_H\cup S$ and $S_H\cap S = \gamma$.
\end{center}

Condition (3) above ensures that this foliation  $\mathcal{F}$ on $\hat{S}$ matches in a $C^1$ way with the partial foliation $\mathcal{F}$ on the grafted rectangles on the other side of $\gamma$, producing a ``horizontal" foliation of $P$. Each leaf of $\mathcal{F}$ in $P$ has a height (in the interval $[-H,H]$) obtained by following it to meet the geodesic $\gamma$, and considering the height of that point of intersection. \\

We define the map of the pentagonal piece $P$  by defining it on the two pieces $S_H$ and $S$: the next lemma deals with the former, and for the piece $S$ we already have Lemma \ref{lem:rectModel}, these are put together in Lemma \ref{lem:hexmap}.

\begin{lem}[Straightening $S_H$]\label{lem:trunc}
There is a quasiconformal homeomorphism $f:S_H\to \mathbb{R}^2$ such that \newline
(i) Each leaf of ${\mathcal{F}}$ is mapped isometrically to a horizontal line segment.\newline
(ii) For all $y$, the left endpoint of the image of the leaf at height $y$ is $(0,y)$.\newline
(iii) The quasiconformal distortion of $f$ is $1 + O(\epsilon)$ at all points of $S_H$ at height $\lvert h\rvert > D = \ln(1/\epsilon)$.\newline
(iv) $\gamma$ is mapped to a graph of a function $g$ over the $y$-axis that is  $C^1$ except at $0$, and is almost-vertical and $\lvert g(y)\rvert = O(\epsilon)$ at points with $\lvert y\rvert > D$. Moreover, $\sup\limits_{y} g(y) = g(0)<1$.
\end{lem}
\begin{proof}
We map the part $S_H^+$ above $a$ (of nonnegative height) and extend to the whole of $S_H$ by reflection. We denote this map of $S_H^+$ by $f_+$.\newline

Note that (i) and (ii) uniquely determines $f_+$, and ensures it is injective. The fact that the foliation is $C^1$ and the lengths of the leaves is $C^1$ in (positive) height ensures that $f_+$ is $C^1$, and is hence a homeomorphism to its image. (iii) follows from Lemma \ref{lem:thinflipped}.\newline
The fact that $f_+$ is isometric on the leaves implies that the function $g$ that describes the image of $\gamma$ as a graph is the length-function of the leaves. This is $C^1$, except at $0$, by property (2) of the foliation $\mathcal{F}$. Part (iv) again follows from Lemma \ref{lem:thinflipped} (a calculation similar to Lemma \ref{lem:qcthin}). In fact, $\lvert g(y)\rvert \to 0$ exponentially as $\lvert y\rvert \to \infty$. The last statement of the lemma follows from the fact that the length function was chosen to be decreasing for increasing height (or decreasing height, by reflection) and the length at height $0$ is the length of the geodesic arc $a$, which is $\ln(\sqrt 3) \approx 0.54$.
\end{proof}

\begin{lem}[Map of a pentagonal piece]\label{lem:hexmap}
There exists a quasiconformal map $f$ from $P$ to a euclidean rectangle of height $2H$ and width $2$, such that\newline
(i) $f$ is height-preserving.\newline
(ii) On the top and bottom sides, $f$ is $\epsilon$-almost-isometric.\newline
(iii) The quasiconformal distortion is $1 + O(\epsilon)$ on points of $P$ outside a $D$-neighborhood of the $2\pi/3$-angled vertex of $P$ (Recall $D=\ln(1/\epsilon)$).
\end{lem}
\begin{proof}
The map $f$ restricted to the truncated sector $S_H\subset P$ shall be the map from the previous lemma (Lemma \ref{lem:trunc}). Note that $f$ satisfies (i), (ii) and  (iii) on $S_H$. In fact it is isometric on the top and bottom horocyclic edges, which is stronger than (ii).\newline

Recall also that this map sends $\gamma$ to a graph of a function $g$ that satisfies:\newline
(1) It is $C^1$ except at one point.\newline
(2) It is almost-vertical for height more than $D$.\newline
(3) $\lvert g(y)\rvert = O(\epsilon)$ for $\lvert y\rvert >D$.\newline
(4)  $\sup\limits_{y} g(y) = g(0)<1$.\newline

\begin{figure}
  \centering
  \includegraphics[scale=0.45]{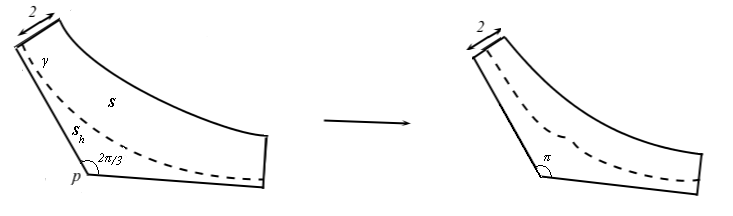}\\
  \caption{The map for a pentagonal piece (shown on the left) to a euclidean rectangle. The region $S_H$ to the left of $\gamma$ is mapped by Lemma \ref{lem:trunc} and the region $S$ to its right by Lemma \ref{lem:rectModel}, adjusted by an application of Lemma \ref{lem:qcstraight}. The map on $\gamma$ is the same as they both preserve height.}
\end{figure}

We can consider this image of $\gamma$ as the graph of the function $w(y) = 2-g(y)$ on a segment on the vertical line $y=2$ (by abuse of orientation the `positive' side of the vertical segment is to its left).  Now by Lemma \ref{lem:rectModel} we map the grafted rectangle $S$ to a euclidean rectangle $S^\prime$ of height $2H$ and width $2$ by an almost conformal height-preserving  map $f_1$. By Lemma \ref{lem:qcstraight} there is a height-preserving quasiconformal map $f_2$  of this rectangle to the planar region bounded by a vertical segment on the line $y=2$ and the graph of $w(y)$ over this segment (which lies to its left).\newline

The map $f$ restricted to $S\subset P$ shall be the composition $f_2\circ f_1$. By properties (2) and (3) of the image of $\gamma$, the appropriate conditions of Lemma \ref{lem:qcstraight} are satisfied and $f_2$ is almost-conformal for  all points of height more than $D$. Since $f_1$ is almost-conformal everywhere, the composition satisfies (iii). Since both $f_1$ and $f_2$ are height-preserving, (i) is satisfied. Finally, it can be checked that (ii) holds because of Corollary \ref{cor:rectnoskew} and Lemmas \ref{lem:noskew0} and \ref{lem:noskew1}.\newline

Thus the map $f$ is defined on $S_H$ and $S$, and hence on their union $P$ (since it is height-preserving on both they match along $\gamma$). Another application of Lemma \ref{lem:noskew1} implies that it satisfies (ii). The usual property of quasiconformal maps (Property B at the beginning of \S 4.2) implies that it is quasiconformal everywhere, and (iii) is satisfied since it is satisfied on both $S_H$ and $S$.
\end{proof}

\subsection{Mapping the grafted surface}

In this section we shall use the decomposition $\mathcal{D}$ of $X_t$ into pentagons  $\{P_j\}_{1\leq j\leq 3m}$ and rectangles $\{R_i\}_{1\leq i\leq n}$, as described at the end of \S 4.1. \newline

Recall from \S3.1 that $\hat{X_t}$ is the singular flat surface obtained by collapsing the ideal triangle components of $X_t\setminus \lambda$ along the leaves of $\mathcal{F}$ and preserving the transverse measure (one could use a Cantor function, as in \cite{CassBl}). Recall that the triangular regions in the complement of the partial foliation $\mathcal{F}$ collapse to the singularities (cone points of  angle $3\pi$). \newline

The pentagons in the decomposition of $X_t$ are mapped by the collapsing map to euclidean rectangles $\{S_j\}$ (of euclidean width $2$), which together with the collapsed images $\{R^\prime_i\}$ of the trimmed-rectangles form a rectangular decomposition of $\hat{X_t}$ with the same combinatorics as $\mathcal{D}$.\newline

Let $K$ denote the $1$-skeleton of the decomposition $\mathcal{D}$ of $X_t$, and  let $\hat{K}$ the corresponding $1$-skeleton on the singular flat surface $\hat{X_t}$.  An edge $E\subset K$ is either \textit{horizontal} (if it is formed of segments of the horizontal foliation) or \textit{vertical}, and we denote by $K_H$ the collection of horizontal edges.\newline

When the hyperbolic part is collapsed, a segment of a leaf of $\mathcal{F}$ of total width $W$ collapses to a segment of width $W_e$ (which, recall, is the \textit{euclidean} width of the segment). In particular, since the hyperbolic width is $O(\epsilon)$ for every horizontal edge $E$, the corresponding edge $\hat{E}$ of $\hat{K}$ differs in length by $O(\epsilon)$.\newline

For each pentagonal piece $P_j$ consider the map $f^P_j:P_j\to S_j$ which is the map $\bar{f}$ obtained from Lemma \ref{lem:hexmap}. Let $\bar{f^P_j}$ denote the restriction of the map to the horizontal edges of $\partial P_j$. These are $\epsilon$-almost-isometric (part (ii) of Lemma \ref{lem:hexmap}). Choose a map $\bar{g}:K_H\to \hat{K_H}$ that is equal to $\bar{f^P_j}$ for each horizontal edge of a pentagonal piece, and affine on each remaining horizontal edge. By the above observation on the difference of lengths of the edges, and Lemmas \ref{lem:noskew1} and \ref{lem:noskew2}, $\bar{g}$ satisfies an $M\epsilon$-almost-isometry condition on each horizontal edge of the $1$-skeleton for some $M$ that depends on the genus (the number of sub-edges of each horizontal edge is determined by the number of rectangles and pentagons in the decomposition that depends only on the genus). We can henceforth absorb that constant in the $O(\epsilon)$ term in the definition of almost-isometry, and refer to the above map as being $\epsilon$-almost-isometric. \newline

For each rectangle $R_i$ in the decomposition, consider the map $f^R_i:R_i\to R^\prime_i$ obtained from Lemma \ref{lem:rectModel}. Let $\bar{f^R_i}$ denote the restriction of the map to $\partial R_i$. Choose a map $\bar{h_i}:\partial R_i \to \partial R_i$ that is isometric on the vertical edges and agrees with $\bar{g}\circ \bar{f^R_i}^{-1}$ on the horizontal edges. By Corollary \ref{cor:rectnoskew}, and Lemma \ref{lem:noskew0}, this map is $\epsilon$-almost-isometric on the horizontal edges. The modulus of $R^\prime_i$ is greater than $1$ by Lemma \ref{lem:widerect}, so we can apply Lemma \ref{lem:qcnoskew} and extend $\bar{h_i}$ to an almost-conformal self-map $h_i$ of $R^\prime_i$.  The map $h_i\circ f^R_i$ now maps $R_i$ to $R^\prime_i$ such that the map agrees with $\bar{g}$ on $K_H$. \newline

The maps $\{h_i\circ f^R_i\}$  and $\{f^P_i\}$ agree on each vertical edge $E$ since they are height-preserving. Hence these maps  agree on the $1$-skeleton $K$ and form a continuous map from $X_t$ to $\hat{X_t}$ that is quasiconformal on each piece $P_i$ and $R_i$ , and is hence quasiconformal everywhere.\newline

Each map from the collection $\{h_i\circ f^R_i\}$ is almost-conformal (since $h_i$ is almost-conformal from the above construction, and $f^R_i$ is almost-conformal by Lemma \ref{lem:rectModel}). Each map $f^P_i$ is almost-conformal away from a set of diameter $O(\ln(1/\epsilon))$, by part (iii) of Lemma \ref{lem:hexmap}.\newline

We summarize this discussion in the following lemma.

\begin{lem}[Map of grafted surface]\label{lem:Map}
There exists a $T>0$ such that for all $t>T$ there is a quasiconformal homeomorphism $f:X_t \to \hat{X_t}$ such that the quasiconformal distortion is $1+ O(\epsilon)$ away from finitely many simply-connected subsets $K_1,K_2,\ldots K_m \subset X_t$ of diameter $O(\ln(1/\epsilon))$ in the Thurston metric. Moreover, each subset $K_j$ is contained in a simply connected set $D_j$ of diameter $\frac{1}{\epsilon^2}$.
\end{lem}

Each $2\pi/3$-angled vertex of a pentagonal piece in the decomposition $\mathcal{D}$ is a centre of an ideal triangle complement of the lamination $\lambda$ (in fact each such center is common to three pentagons). The subsets $K_1,K_2,\ldots K_m \subset X_t$ in the lemma are the $O(\ln(1/\epsilon))$-neighborhoods of the finitely many centres of the ideal triangle complements of $\lambda$, by part (iii) of Lemma \ref{lem:hexmap}. The last statement follows from Lemmas \ref{lem:thinrect} and \ref{lem:widerect}: the three rectangles adjacent to a truncated ideal triangle $T_j$ are of height and width at least $\frac{1}{\epsilon^2}$ when the grafting time $t>T$ is sufficiently large, hence one can embed a disk $D_j$ of diameter $\frac{1}{\epsilon^2}$ centered at the center of $T_j$, containing $K_j$.

\begin{figure}
  \centering
  \includegraphics[scale=0.5]{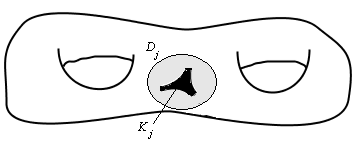}\\
  \caption{A typical pair $(D_j,K_j)$ on the grafted surface. The map in Lemma \ref{lem:Map} is almost-conformal outside $K_j$ (shaded darker). The annular region $D_j\setminus K_j$ (shaded lighter) has large modulus (Lemma \ref{lem:largemod}). }\label{fig:img}
\end{figure}

\begin{lem}\label{lem:largemod}
For large enough $t$ (as in Lemma \ref{lem:Map}), each annular region $D_j\setminus K_j$ for $1\leq j\leq m$ where $D_j$ and $K_j$ are the simply-connected subsets as above, has modulus greater than $2\pi\ln(1/\epsilon)$.
\end{lem}
\begin{proof}
Consider the annular region that is the image of $f(D_j\setminus K_j)$ on the singular flat surface $\hat{X_t}$. From the construction of $f$ (in particular its property of being height-preserving and almost-isometric along $\mathcal{F}$) it follows that one can embed a flat annulus $A_j$ of large modulus in $f(D_j\setminus K_j)$, In fact, since the diameter of $f(D_j)$ is greater than $\frac{1}{\epsilon^2}$ and the diameter of $f(K_j)$ is at most $O(\ln\frac{1}{\epsilon})$, the modulus of $A_j$   satisfies
\begin{equation*}
mod(A_j) = \frac{1}{3\cdot 2\pi}\ln\frac{R_{out}}{R_{inn}} \geq A\ln\frac{1/\epsilon^2}{\ln\frac{1}{\epsilon}} + B > 4\pi\ln\frac{1}{\epsilon}
\end{equation*}
for $\epsilon$ sufficiently small ($A$ and $B$ are some positive constants). \newline

Since $f$ is almost-conformal on $D_j\setminus K_j$, we have an embedded annulus $f^{-1}(A_j)$ of large modulus in $D_j\setminus K_j$ (in the Thurston metric). In particular, the modulus 
\begin{equation*}
mod(D_j\setminus K_j) >  \frac{4\pi}{1 + \epsilon}\ln\frac{1}{\epsilon} > 2\pi\ln\frac{1}{\epsilon}
\end{equation*} \end{proof}

\subsection{Modifying the map to almost-conformality}

In this section we modify the map $f:X_t\to \hat{X_t}$ from Lemma \ref{lem:Map} such that it is almost-conformal. To do this we shall redefine the map in the subsets $D_1,D_2,\ldots D_m$. The crucial fact that allows this modification is that the lack of almost-conformality for $f$ is contained in the subsets $K_j$, $1\leq j\leq m$ which have diameter much smaller than the diameter of $D_j$.

\subsubsection{A quasiconformal extension lemma}
Let $\mathbb{D}$ be the unit disk in $\mathbb{C}$ and let $B_r$ be a closed ball of radius $r$ about the origin.\newline

The following result is probably well-known to experts, however it does not seem to be readily available in the literature on the subject. A proof is provided in Appendix A.
\begin{lem}\label{lem:qclem}
For any $\epsilon>0$ sufficiently small and any $0\leq r\leq \epsilon$ if $f:\mathbb{D}\to \mathbb{D}$ satisfies\newline
(1) $f$ is a quasiconformal map,\newline
(2) The quasiconformal distortion is $(1+ C\epsilon)$ on $\mathbb{D}\setminus B_r$,\newline
then the map $f$ extends to a $(1 + C^\prime\epsilon)$-quasisymmetric map of the boundary, where $C^\prime$ is a constant depending only on $C$.
\end{lem}

An immediate consequence of the Ahlfors-Beurling extension is
\begin{cor}\label{cor:qcext}
Let  $\epsilon>0$ be sufficiently small, $r\leq \epsilon$ and let $f:\mathbb{D}\to \mathbb{D}$ satisfy (1) and (2) as in the previous lemma. Then there exists an almost-conformal map $g:\mathbb{D}\to \mathbb{D}$ such that $f_{\vert \partial \mathbb{D}} = g_{\vert \partial \mathbb{D}}$.
\end{cor}

We shall use the above for modifying a map between Riemann surfaces:

\begin{cor}\label{cor:qcext2}
Let $\Sigma,\Sigma^\prime$ be homeomorphic Riemann surfaces. Let $f:\Sigma\to \Sigma^\prime$ be a quasiconformal map and $K\subset D\subset \Sigma$ be concentric embedded disks such that\newline
(i) The modulus of the annulus $D\setminus K$ is at least $2\pi\ln\frac{1}{\epsilon}$.\newline
(ii) $f$ is almost-conformal on $\Sigma \setminus K$.\newline
Then there is a quasiconformal map $g:\Sigma\to \Sigma^\prime$ such that $f\vert_{\Sigma\setminus D} = g\vert_{\Sigma\setminus D}$ and $g$ is almost-conformal on $D$.
\end{cor}
\begin{proof}
Let $\phi:D\to \mathbb{D}$ and $\psi:f(D)\to \mathbb{D}$ be uniformizing maps to the unit disk, normalized such that the centers are taken to $0\in \mathbb{D}$. Then by (i) above $\phi(K)$ and $\psi(f(K))$ have diameter $O(\epsilon)$  (this is an application of equation (\ref{eq:ann}) in Lemma \ref{lem:qconf}) and contain $0\in \mathbb{D}$. The map $g=\psi\circ f\circ \phi^{-1}:\mathbb{D}\to \mathbb{D}$ satisfy the requirements of Corollary \ref{cor:qcext} and hence can be replaced by an almost-conformal map $h:\mathbb{D}\to \mathbb{D}$ that has the same map as $g$ on $\partial \mathbb{D}$. We replace $f:D\to f(D)$ by the almost-conformal map $\psi^{-1}\circ h\circ \phi:D\to f(D)$. This restricts to the same map as $f$ on $\partial D$. Together with the map $f$ on $\Sigma \setminus D$ it defines a continuous map of $\Sigma$ to $\Sigma^\prime$. This map  is quasiconformal on $D$ and $\Sigma \setminus D$, and is hence quasiconformal, since $\partial D$ is a measure zero set.
\end{proof}

\subsection{Proof of Proposition \ref{prop:maxcase}}

\begin{proof}[Proof of Proposition \ref{prop:maxcase}]

We start with the map from Lemma \ref{lem:Map} and modify it by applying Corollary \ref{cor:qcext2} (taking $\Sigma = X_t$, $\Sigma^\prime = \hat{X_t}$, $D=D_j$, $K=K_j$),  for each $1\leq j\leq m$ in succession. This is possible since each annulus $D_j\setminus K_j$ has modulus at least $2\pi\ln\frac{1}{\epsilon}$ by Lemma \ref{lem:largemod} (so (i) of Cor. \ref{cor:qcext2} holds). The final map $f:X_t\to \hat{X_t}$ is almost-conformal on $D_1\cup D_2 \cup \cdots D_m$ and agrees with the original $f$, and is hence almost-conformal  on the complement $X_t \setminus D_1\cup D_2 \cup \cdots D_m$. The property of almost-conformality extends across the measure zero set consisting of the union of the $\partial D_j$s.\newline

This completes the construction of an almost-conformal map $f:X_t\to \hat{X_t}$, for all $t$ sufficiently large.\newline

The singular flat surface $\hat{X_0}$ has a horizontal foliation, and a vertical foliation which is measure equivalent to $\lambda$. Now $\hat{X_t}$ can be obtained from $\hat{X_0}$ by scaling the horizontal foliation on $\hat{X_0}$ by a factor of $t$, and keeping the vertical foliation the same, which is conformally equivalent to scaling the horizontal foliation by a factor of $\sqrt t$ and the vertical foliation by a factor of $1/\sqrt t$. The surface $\hat{X_t}$ thus lies on the Teichm\"{u}ller ray from $\hat{X_0}$ determined by $\lambda$, at a distance of $\frac{1}{2}\ln t$.\newline

Since our choice of $\epsilon>0$ throughout was arbitrary, this shows the grafting ray based at $X=X_0$ determined by the lamination $\lambda$ is asymptotic  to the Teichm\"{u}ller ray based at $Y=\hat{X_0}$.
\end{proof}

The proofs of Corollaries 1.2 and 1.3 only require Theorem 1.1 in the arational case, hence we provide them here:

\begin{proof}[Proof of Corollary 1.2]
Pick any Teichm\"{u}ller ray determined by such an arational $\lambda$. By Proposition \ref{prop:maxcase}, the two grafting rays are both asymptotic to it, and hence to each other.\\

When $\lambda$ is also uniquely-ergodic, then for any other choice of basepoint $Y$, the Teichm\"{u}ller ray $Y_t$ determined by $\lambda$ is asymptotic to the above Teichm\"{u}ller ray, by the result of Masur (\cite{Mas}), and hence by the triangle inequality, the grafting ray based at $X$ is asymptotic to the Teichm\"{u}ller ray $Y_t$.
\end{proof}

\begin{proof}[Proof of Corollary 1.3]
Arational laminations form a full measure set in $\mathcal{ML}$ (with respect to the Thurston measure on $\mathcal{ML}$). It is known that for a generic choice of such an arational $\lambda$ and any choice of basepoint the corresponding Teichm\"{u}ller ray is dense in moduli space (this follows from the ergodicity of the Teichm\"{u}ller geodesic flow, proved in \cite{MasErg} - see \cite{Mas2} for explicit examples of such rays). By Proposition \ref{prop:maxcase}, a grafting ray determined by a generic $\lambda$ is then asymptotic to a dense Teichm\"{u}ller ray, and is hence itself dense.
\end{proof}

\textit{Remark.} It was pointed out to the author by Jayadev Athreya that it is easy to show that Theorem \ref{thm:thm1} implies that such a grafting ray is in fact equidistributed in moduli space. A subsequent article shall investigate some further dynamical and measure-theoretic features of the Teichm\"{u}ller geodesic flow inherited by grafting flowlines.

\section{The multi-curve case}
In this section we prove Theorem 1.1 in the case when $\lambda$ is a multi-curve (Proposition \ref{prop:scc}) following the outline in \S 3.2. 

\subsection{A quasiconformal interpolation}
We start with a lemma about quasiconformal maps that shall be useful later. Throughout, $\mathbb{D}$ shall denote the unit closed disk on the complex plane, and $B_r$ shall denote the closed disk of radius $r$ centered at $0$. We show that a conformal map defined on $\mathbb{D}$ can be adjusted to be the identity near $0$, without too much quasiconformal distortion.

\begin{lem}\label{lem:gfix}
Let $g:\mathbb{D}\to g(\mathbb{D})\subset \mathbb{C}$ be a univalent conformal map such that $g(0)=0$ and $g^\prime(0)=1$. Then for any $\epsilon>0$ (sufficiently small) there exists an $0<s<1$ and a map $G:\mathbb{D}\to g(\mathbb{D})$ such that\newline
(1) $G$ is $(1 + \epsilon)$-quasiconformal.\newline
(2) $G$ restricts to the identity map on $B_{s}$.\newline
(3) $G\vert_{\partial \mathbb{D}} = g\vert_{\partial \mathbb{D}}$.
\end{lem}
\begin{proof}
Since $g(0)=0$ and $g$ is conformal, there exists an an expansion
\begin{equation*}
g(z) = z + a_2z^2 + a_3z^3 +\cdots =  z + \psi(z)
\end{equation*}
where $a_i\in \mathbb{C}$ for $i\geq 2$.\newline

We shall have $s=\epsilon$. Let $\phi_\epsilon:[0,1]\to [0,1]$ be a smooth bump function such that\newline
(1) $\phi_\epsilon(t) = 0$ for $0\leq  t \leq \epsilon$.\newline
(2) $\phi_\epsilon(t) = 1$ for $2\epsilon\leq  t \leq 1$ .\newline
(3) $\lvert \phi_\epsilon^\prime(t)\lvert = O(1/\epsilon)$ for all $t\in [\epsilon,2\epsilon]$.\newline
and define
\begin{equation*}
\psi_\epsilon(z) = \phi_\epsilon(\lvert z\rvert)\psi(z)
\end{equation*}
for all $z\in \mathbb{D}$.\newline 

Define the map $G: \mathbb{D} \to \mathbb{C}$ as
\begin{equation*}
G(z) = z + \psi_{\epsilon}(z)
\end{equation*}
for $z\in \mathbb{D}$.\newline

From the Koebe distortion theorem (see for example Theorem 1.3 of \cite{Pom}) we have
\begin{equation*}
\frac{1}{(1 + \lvert z \rvert)^2} \leq \left\lvert \frac{g(z)}{z}\right\rvert \leq \frac{1}{(1 - \lvert z \rvert)^2}
\end{equation*}
and that implies
\begin{equation*}
\left\lvert \frac{g(z)}{z} -1 \right\rvert \leq \frac{4\lvert z\rvert}{(1 - \lvert z\rvert^2)^2}
\implies
\lvert \psi(z)\rvert \leq \frac{4\lvert z\rvert^2}{(1 - \lvert z\rvert^2)^2}
\end{equation*}

For $\epsilon\leq \lvert z\rvert \leq 2\epsilon$ we therefore have 
\begin{equation*}
\lvert \psi(z)\rvert \leq C\epsilon^2
\end{equation*}
for some universal constant $C$ (we know $\epsilon$ is sufficiently small).\newline
From this and (3) above it is easy to check that 
\begin{equation*}
\lvert \partial_{\bar{z}} G\rvert = \lvert \partial_{\bar{z}} \psi_{\epsilon}\rvert = O(\epsilon)
\end{equation*}
and 
\begin{equation*}
\lvert \partial_{z} G\rvert = \lvert 1 +  \partial_{z} \psi_{\epsilon}\rvert = 1 + O(\epsilon)
\end{equation*}
for each $z\in\mathbb{D}$.\newline

$G$ is hence a diffeomorphism of quasiconformal dilatation of $1 + O(\epsilon)$ that restricts to the identity map on $B_{\epsilon}$ and to $g$ on $\partial \mathbb{D}$ as required.
\end{proof}

\textit{Remark.} By conjugating by the dilation $z\mapsto (1/r)z$ the above result holds (for some $0<s< r$)  if the conformal map $g$ is defined only on $B_r\subset \mathbb{D}$.

\subsection{Proof of the multi-curve case}

Let $\lambda$ be a multicurve, namely a collection of (weighted) disjoint simple closed geodesics $\{ \gamma_1,\gamma_2,\ldots,\gamma_n\}$ on the closed hyperbolic surface $X$, with weights $c_i>0$ and  lengths $l_i$, where $1\leq i\leq n$.\newline

Let $X^\infty$ denote the infinitely grafted surface, obtained by cutting $X$ along $\lambda$ and gluing a semi-infinite euclidean cylinder at each of the resulting $2n$ boundary components. This resulting Riemann surface can be thought of as the conformal limit of the grafting ray $X_t =gr_{t\lambda}X$. (Recall that a time-$t$ grafting along a simple closed curve $\gamma$ inserts a euclidean cylinder of width $t$ at $\gamma$).\newline

We shall use the following result due to Strebel (\cite{Streb}):

\begin{prop}\label{prop:streb} Let $\Sigma$ be a Riemann surface of genus $g$, and $x_1,x_2,\ldots x_m$ be marked points on $\Sigma$ such that $2g-2+m>0$ . Then for any collection of real numbers $p_1,p_2,\ldots,p_m$ there exists a quadratic differential $\phi$ on $\Sigma$ such that \newline
(i) $\phi$ is holomorphic on $\Sigma\setminus \{x_1,\ldots,x_m\}$.\newline
(ii) $\phi$ has a double pole at $x_i$ with residue $-(\frac{p_i}{2\pi})^2$, for each $1\leq i\leq m$.\newline
(iii) All horizontal trajectories of $\phi$ are closed, and they foliate $m$ annular domains which are disks with punctures at $x_1,\ldots,x_m$.
\end{prop}

Applying the above proposition with $\Sigma = X^\infty$, we can obtain such a \textit{Jenkins-Strebel differential} $\phi$ on a Riemann surface conformally equivalent (as marked conformal structures in $\mathcal{T}_{g,2n}$) to $X^\infty$, with $n$ pairs of marked points, each pair having residue $-(\frac{l_i}{2\pi})^2$. We denote this surface equipped with the quadratic differential metric as $Y^\infty$: this then is a singular flat surface comprising $n$ pairs of infinite euclidean cylinders, each pair having circumference $l_i$.  Let $g$ be the conformal map from $X^\infty$ to $Y^\infty$ that preserves the marking.\newline

\begin{figure}
  \centering
  \includegraphics[scale=0.35]{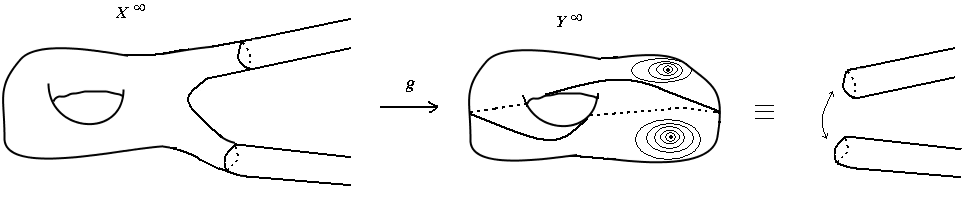}\\
  \caption{The surface $X^\infty$ on the left is the ``infinitely" grafted surface. The conformally equivalent singular flat surface $Y^\infty$ has a quadratic differential metric with a pair of double poles, and metrically it is equivalent to two semi-infinite euclidean cylinders glued along the boundary.}
\end{figure}

Let $Y$ denote the surface obtained from $Y^\infty$ by truncating the infinite cylinders and gluing up the pairs (of matching lengths)  so that they give euclidean cylinders of circumference $l_i$ and height $c_i$, in the homotopy class of $\gamma_i$, for each $1\leq i\leq m$. This will be the basepoint of the Teichm\"{u}ller ray $Y_t$ that we shall show is asymptotic to the grafting ray $X_t$. Recall that the surface $Y_t$ is obtained from $Y$ (which is also $Y_0$ in our notation) by stretching along the horizontal foliation, and in particular the euclidean cylinders on $Y_t$ have height $c_ie^{2t}$.\newline

\textit{Notation.} For a semi-infinite cylinder (homeomorphic to $S^1\times [0,\infty)$) denoted by $C$ ,we denote by $C_{\geq h}$ (resp. $C_{\leq h}$) the infinite subcylinder of all points of $C$ of height greater (resp. lesser) than $h$. 

\begin{lem}\label{lem:conf1}
Let $C, C^\prime$ be two semi-infinite Euclidean cylinders of the same circumference, and let $f:C\to C^\prime$ be a conformal map which is a homeomorphism onto its image. Then for any $H_0>0$  there exists an $H_1>H_0$ and a map $F:C\to f(C)\subset C^\prime$ such that\newline
(1) $F$ is $(1 + \epsilon)$-quasiconformal.\newline
(2) $F$ is isometric on $C_{\geq H_1}$.\newline
(3) $F$ restricts to $f$ on $C_{\leq H_0}$.\\
(Here $H_1$ depends on $\epsilon$ and $H_1\to \infty$ as $\epsilon\to0$.)
\end{lem}
\begin{proof}
The cylinders $C$ and $C^\prime$ are conformally equivalent to the punctured unit disk $\mathbb{D}^*$ by a conformal map $\phi$ that takes $\infty$ to $0$, and maps the round circle at height $h$ (for each $0\leq h< \infty$) to a circle of radius $r(h) = e^{-2\pi h}$ centered at the origin. Thus the conformal map $f$ conjugates to a conformal map $g = \phi\circ f\circ \phi^{-1}$  from the punctured disk to itself, which can be extended to a conformal map of $\mathbb{D}$ into itself, such that $g(0)=0$ and $g^\prime(0) = c$.\newline
We can now apply Lemma \ref{lem:gfix} to the rescaled conformal map $(1/c)g$ to obtain an almost-conformal map $G$ that restricts to a dilatation $z\mapsto cz$ on $B_s$ for some $0<s<1$ and agrees with $g$ on $\partial \mathbb{D}$. Since a dilatation conjugates back to an isometric translation of the semi-infinite cylinders, the map $F = \phi^{-1}\circ G\circ \phi$ satisfies (1) to (3).
\end{proof}

\begin{figure}
  \centering
  \includegraphics[scale=0.35]{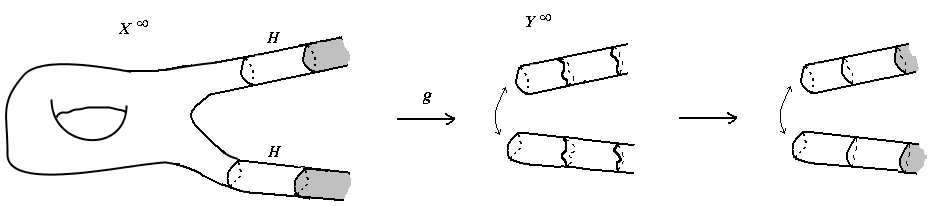}\\
  \caption{Lemma \ref{lem:conf1} allows one to adjust $g$ such that it isometric on the circle at height $H$. Discarding the shaded regions and gluing up along the truncating circles gives a grafted surface $X_t$ on the left and the surface along the Teichm\"{u}ller ray on the right. }
\end{figure}

\begin{prop}\label{prop:scc}
For any (sufficiently small) $\epsilon>0$ there exists a (sufficiently large) $T>0$ such that for every $t>T$ there is a $(1 + \epsilon)$-quasiconformal map from the grafted surface $X_t$ to a singular flat surface $Y_s$ for some $s$. 
\end{prop}
\begin{proof}
We start with the conformal map $g:X^\infty \to Y^\infty$. Consider its restriction to one of the semi-infinite euclidean cylinders $C$, and let $C^\prime$ be the corresponding infinite cylinder on $Y^\infty$ such that $g(C)\cap C^\prime \neq \phi$. By properness, for $H_0$ sufficiently large, the image under $g$ of the $C_{\geq H_0}$ will be strictly contained in $C^\prime$. Applying Lemma \ref{lem:conf1} we have some $H_1>H_0$ and a map $F_C$ on the truncated cylinder $C_{\leq H_1}$ that  agrees with $g$ on $C_{\leq H_1}$ and is isometric on the circle at height $H_1$.\newline

Repeating this for each infinite cylinder on $X^\infty$, we obtain truncations of each and almost-conformal maps such that together we have an almost-conformal map $F$ to $Y^\infty$, with each of its cylinders also truncated at a circle at some height. The map $F$ is isometric on the boundary circles, and in particular its restrictions to truncated paired cylinders agree on the truncating round circles. On gluing these maps we obtain an almost-conformal map from the truncated $X^\infty$ glued along the boundaries of the paired cylinders, to the truncated $Y^\infty$ glued along the boundaries of its paired cylinders. By definitions of the surfaces, the latter is $Y_s$ for some $s$, and the former is $X_T$ for some $T$. By choosing to truncate at  higher heights (greater than $H_1$ as in $C$ above) one can obtain a almost-conformal maps from $X_t$ for any $t>T$, to a surface along the Teichm\"{u}ller ray starting at $Y$.\\

This ``almost-conformal" map is $(1 + C\epsilon)$-quasiconformal for some universal constant $C>0$ (see for example the lemmas in \S5.1), and so by choosing a smaller $\epsilon$ (and consequently a larger $T$ for the above construction) on can obtain a $(1+\epsilon)$-quasiconformal map, as stated.
\end{proof}
\vspace{0.13in}

Recall that the surfaces $X_t$ form the grafting ray determined by the multicurve $\lambda$, and the surfaces $Y_s$ lie along the Teichm\"{u}ller ray determined by $\lambda$ and with basepoint $Y=Y_0$. Thus together with Proposition \ref{prop:maxcase} this concludes the proof of Theorem 1.1.

\section{Proof of Theorem 1.4}

In this section we prove the following theorem stated in \S1:

\begin{thm14}
Let $\mathcal{P}_\rho$ be the set of complex projective structures on a surface $S_g$ with a  fixed holonomy $\rho\in Rep(\pi_1(S_g),PSL_2(\mathbb{C}))$. Then for any Fuchsian representation $\rho$, the projection of $\mathcal{P}_\rho$ to $\mathcal{M}_g$ has a dense image.
\end{thm14}

Recall that $\pi:\mathcal{T}_g\to \mathcal{M}_g$ denotes the projection to moduli space, and $\mathcal{P}_\rho \subset \mathcal{P}(S_g)$ denotes the set of complex projective structures that have holonomy $\rho$. Let $p:\mathcal{P}(S_g) \to \mathcal{T}_g$ be the ``forgetful" projection.\\

The following is a consequence of the work of Goldman (see \S 2.14 in \cite{Gold}):\\

\textbf{Theorem}.(\cite{Gold}) \textit{Let $X$ be a hyperbolic surface, and let $\rho:\pi_1(S_g)\to PSL_2(\mathbb{R})$ be the corresponding Fuchsian representation. Then grafting along a multi-curve $\gamma$ on $X$ with weight $2\pi n$, for some integer $n\geq 1$, preserves Fuchsian holonomy. In fact, any other element $Y\in\mathcal{P}_\rho$ is obtained by such a grafting.}\\

\textit{Remark.} We shall refer to grafting a hyperbolic surface along a multi-curve with weight $2\pi$ as  \textit{integer}-grafting or $2\pi$-grafting.\\

It immediately follows that Theorem 1.4 is a consequence of  the following proposition:

\begin{prop}\label{prop:prop0}
For any $X\in\mathcal{T}_g$, the set of integer graftings $\{\pi(gr_{2\pi\gamma}X)\vert$ $\gamma\in \mathcal{S}\}$ is dense in $\mathcal{M}_g$.
\end{prop}

\textit{Plan of the proof.} Fix any $X\in \mathcal{T}_g$ and pick an arbitrary $Y\in \mathcal{M}_g$ and a sufficiently small $\epsilon>0$.  To establish Proposition \ref{prop:prop0} it is enough to show that for this choice we have:

\begin{prop}\label{prop:prop1}
There exists a $\gamma\in \mathcal{S}$ such that $d_\mathcal{T} (\pi(gr_{2\pi\gamma}X), Y) < 2\epsilon$.
\end{prop}

By Corollary 1.3 we have a $\lambda \in \mathcal{ML}$ such that\newline
(1) $\lambda$ is arational, and\newline
(2) The projection of the grafting ray determined by $(X,\lambda)$ to moduli space is dense.\newline
In particular, we have a sequence of times $t_i\to \infty$ such that 
\begin{equation}\label{eq:ti}
d_{\mathcal{T}}(\pi(gr_{2\pi t_i\lambda}(X)),Y)<\epsilon
\end{equation}

\bigskip
The argument for the proof of Proposition \ref{prop:prop1} carried out in \S6.1-5 consists of choosing an appropriate approximation of one of the $2\pi t_i\lambda$ by a multicurve and showing that the corresponding grafted surface is close to the surface obtained by grafting along this multicurve (Lemma \ref{lem:close}). Proposition \ref{prop:prop1} then follows easily from the triangle inequality (\S 6.6).

\subsection{Almost-conformal constructions}

In \S 4.2 we developed the notion of ``almost-isometries"  (see Definition \ref{defn:aisom}) and some related constructions of almost-conformal maps. 
In this section, we weaken the definition by introducing a further bounded additive error $C$ (see Definitions \ref{defn:ecisom} and \ref{defn:ecgood}). This additive error shall be small relative to the other dimensions, however, and will still permit the construction of almost-conformal maps (as in Lemma \ref{lem:ext}).

\begin{defn}\label{defn:ecisom}
A homeomorphism $f$ between two $C^1$-arcs on a conformal surface is an \textit{$(\epsilon,C)$-almost-isometry} if $f$ is continuously differentiable with dilatation $d$ (the supremum of the derivative of $f$ over the domain arc) that satisfies $\lvert d-1\rvert \leq \epsilon$ and such that the lengths of any subinterval and its image differ by an additive error of at most $C$.
\end{defn}

\textit{Remark.} As before, we my sometimes say `$(\epsilon, C)$-almost-isometric' to mean `$(M\epsilon,C)$-almost-isometric for some (universal) constant $M>0$'.\\

The following are analogues of Lemmas \ref{lem:noskew0}, \ref{lem:noskew1} and \ref{lem:noskew2}, and we omit their (easy) proofs:

\begin{lem}\label{lem:lem1}
Let $I_1,I_2$ be arcs of lengths $l_1$ and $l_2$ such that $\lvert l_1-l_2\rvert <C$ and $C/l_1 < \epsilon$. Then the orientation-preserving affine (stretch) map $f:I_1\to I_2$ is $(\epsilon,C)$-almost-isometric.
\end{lem}

\begin{lem}\label{lem:lem2}
Let $f,g:I\to I$ be maps of an arc $I$ that are $(\epsilon,C)$-almost-isometric and $(\epsilon,C^\prime)$-almost-isometric  respectively. Then $f^{-1}$ is $(\epsilon,C)$-almost-isometric, and $f\circ g$ is $(\epsilon,C+C^\prime)$-almost-isometric.
\end{lem}

\begin{lem}\label{lem:lem3}
Let $I=I_1\cup I_2\cup \cdots I_N$ be a partition of the arc $I$ into sub-arcs with disjoint interior. Then any  continuously differentiable map $f:I\to I$ with $(\epsilon,C)$-almost-isometric restrictions to $I_1,\ldots I_N$ is $(\epsilon,NC)$-almost-isometric on $I$.\newline
Conversely, the restriction of an $(\epsilon,C)$-almost-isometry to a sub-arc is also an $(\epsilon,C)$-almost-isometry to its image.
\end{lem}

\begin{defn}\label{defn:ecgood}
A map $f$ between two rectangles is \textit{$(\epsilon,C)$-good} if it is isometric on the vertical sides and $(\epsilon,C)$-almost-isometric on the horizontal sides.
\end{defn}

\textit{Remark.} A ``rectangle" in the above definition refers to four arcs in any metric space intersecting at right angles, with a pair of opposite sides of equal length being identified as \textit{vertical} and the other pair also of identical length called \textit{horizontal}.

\subsubsection{Almost-conformal extension}
The following lemma is a slight generalization of Lemma \ref{lem:qcnoskew}:

\begin{lem}\label{lem:ext}
Let $R_1$ and $R_2$ be two euclidean rectangles with vertical sides of length $h$ and horizontal sides of lengths $l_1$ and $l_2$ respectively, such that  $l_1,l_2 >h$ and $\lvert l_1-l_2\rvert <C$, where $C/h \leq\epsilon$. Then any $(\epsilon,C)$-good map $f:\partial R_1\to \partial R_2$ has a $(1+C\epsilon)$-quasiconformal extension $F:R_1\to R_2$ for some (universal) constant $C>0$.
\end{lem} 
\begin{proof}
The proof follows by rescaling by a factor of $1/h$ and applying Lemma \ref{lem:qcnoskew} to the resulting map between the resulting pair of rectangles.
 \end{proof}

\subsubsection{Finitely grafted rectangle}

Let $R$ be a region in a the hyperbolic plane bounded by two ``vertical" geodesic sides of length $l$ and two ``horizontal" horocyclic sides of length $w$. Assume henceforth that  $l>1/\epsilon$ and $w<\epsilon$.\\

Let $a_1,a_2,\ldots a_k$ be a finite collection of geodesic arcs with endpoints on the horizontal sides, with corresponding weights $w_1,w_2,\ldots w_k$. Then one can obtain a \textit{finitely grafted} rectangle $R^\prime$ by inserting euclidean rectangles in the shape of truncated ``crescents" (see figure) of widths $w_1,w_2,\ldots w_k$ at the arcs $a_1,a_2,\ldots a_k$ respectively.

\begin{figure}
  \centering
  \includegraphics[scale=0.5]{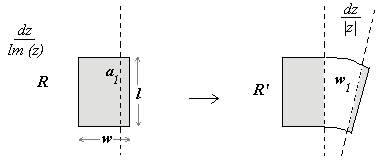}\\
  \caption{Grafting a rectangle $R$ (shown on the upper half plane model) across a single weighted geodesic arc $a_1$ gives a finitely grafted rectangle $R^\prime$ (\S 6.1.2).}
\end{figure}

\begin{lem}\label{lem:fin}
There is an almost-conformal map $f$ from $R^\prime$ to a euclidean rectangle of vertical height $l$ and horizontal width $w_1 + w_2 + \cdots +w_k$. Moreover, $f$ is ($\epsilon$,$\epsilon$)-good on the boundary.
\end{lem}
\begin{proof}
We give a sketch of the argument, and refer to \S 4.2 and \S4.3 for details and similar constructions. We always work in the upper-half-plane model of the hyperbolic plane.\newline
First, we can map the (ungrafted) rectangle $R$ to the euclidean plane by a map that ``straightens" the horocylic foliation across $R$. Since $w<\epsilon$, this straightening map is almost-conformal (Lemma \ref{lem:qcthin}).  It also follows from some elementary hyperbolic geometry that the hyperbolic ``width" between the geodesic arcs on $R$ is a $C^1$-function of the ``height" with $\epsilon$-small derivatives, and so their images under the straightening map are $\epsilon$-almost-vertical.\newline
Next, the truncated ``crescents" are spliced in: their straightening maps to the plane are in fact conformal with rectangular images (Lemma \ref{lem:qceuc}) and hence can be adjusted by almost-conformal maps (Lemma \ref{lem:qcstraight}) to fit with the almost-vertical image arcs above.\newline
This gives a composite map that is almost-conformal with image a rectangle of height $l$ and width $w + w_1 + w_2 + \cdots +w_k$. Since $w<\epsilon$ and $w_1+w_2+\cdots+w_k>1$, one can finally compose by an almost-conformal horizontal affine stretch to a rectangle of width $w_1 + w_2 + \cdots +w_k$ as required.\newline
The statement about the almost-isometry of the sides follows from Lemma \ref{lem:lem1} since prior to the final affine dilatation the map is isometric on the horizontal sides, and the final affine stretch is to a rectangle of width differing by $w <\epsilon$.
\end{proof}

\subsubsection{Smoothing the horizontal sides.} 
In the finitely grafted rectangle $R^\prime$ above, the horizontal sides may not be differentiable arcs since the geodesic arcs $a_1,\ldots a_k$ may not intersect the horizontal sides of $R$ at right angles. However, since the rectangle $R$ prior to grafting is thin ($w<\epsilon$) and long ($l>1/\epsilon$) some elementary hyperbolic geometry implies that the geodesic arcs intersect the horizontal sides at an angle that differs from $\pi/2$ by a quantity bounded by $C\epsilon$ (for some universal constant $C>0$).\\

In $R^\prime$, the horizontal sides can then be ``smoothed" to be $C^1$ by ``trimming" the horizontal sides of each grafted euclidean ``truncated crescent":  each such horizontal segment is replaced by a $C^1$ arc  whose derivatives are $\epsilon$-small, and have specified values at the endpoints that make the entire arc $C^1$. Denote the resulting new ``smoothed" rectangle by $R^{\prime\prime}$.\\

The following lemma can be thought of as a ``vertical" version of the straightening lemma from \S 4.1 (Lemma \ref{lem:qcstraight}). 

\begin{lem}\label{lem:smooth0}
Let $S= [0,w]\times [0,l]$ be a euclidean rectangle and let $S^\prime$ be the region enclosed by the two parallel vertical line segments of $S$ and two arcs which are the graphs $y=g_1(x)$ and $y=g_2(x)$ over the interval $0\leq x\leq w$ on the $x$-axis, where $g_1$ and $g_2$ are $C^1$-functions such that\\
(1) $g_2(x)>g_1(x)>0$,\\
(2) $\lvert g_1^\prime(x)\rvert <\epsilon$ and $\lvert g_1^\prime(x)\rvert <\epsilon$,\\
(3) $\left\vert \frac{l}{g_2(x) - g_1(x)} - 1\right\vert <\epsilon$,\\
for all $0\leq x\leq w$.  Then there exists a $(1+C\epsilon)$-quasiconformal map from $S$ to $S^\prime$ which is $(\epsilon,\epsilon)$-good on the boundary. (Here $C>0$ is some universal constant.)
\end{lem}

\begin{figure}
  \centering
  \includegraphics[scale=0.35]{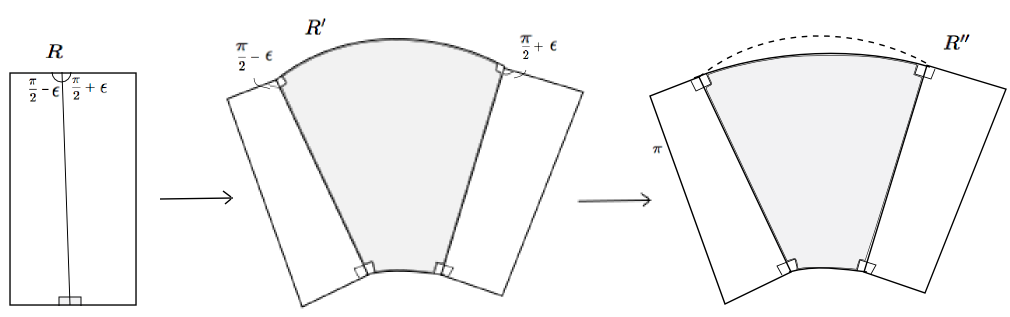}\\
  \caption{Grafting along a geodesic arc intersecting almost, but not at right angles gives a finitely grafted rectangle $R^\prime$ with the top edge having ``corners" (the grafted euclidean region is shown shaded). This can be smoothed to a $C^1$-arc together with an almost-conformal map from $R^\prime$ to the resulting rectangle $R^{\prime\prime}$ (Lemma \ref{lem:smooth}). }
\end{figure}

We omit the proof, which is similar to Lemma \ref{lem:qcstraight}, involving a map that stretches vertically by the right factor along the width of $S^\prime$.\\

By a repeated application of the above lemma on each of the grafted strips, where the graphs $g_1$ and $g_2$ for each are determined by the $C^1$-``trimming",  we have:

\begin{lem}\label{lem:smooth}
There exists an almost-conformal map from $R^\prime$ to $R^{\prime\prime}$ which is $(\epsilon,\epsilon)$-good on the boundary. 
\end{lem}

Moreover, by precomposing with the almost-conformal map $f$ of Lemma \ref{lem:fin}, we have:

\begin{cor}\label{cor:smooth1}
There exists an almost-conformal map from $R^{\prime\prime}$ to a euclidean rectangle of vertical height $l$ and horizontal width $w_1 + w_2 + \cdots +w_k$ which is ($\epsilon$,$\epsilon$)-good on the boundary.
\end{cor}

\subsection{Thickening the train-track}

Recall from \S4.1 that we have the subsurface $\mathcal{T}_{\epsilon}\subset X$ containing the arational lamination $\lambda$ which is its ``$\epsilon$"-train-track neighborhood. From that section, we have a decomposition of the surface into a collection of rectangles $R_1,R_2,\ldots R_n$, and complementary regions which are truncated ideal hyperbolic triangles $T_1,T_2,...T_m$.\\

We now describe a ``thickening" to ensure that $\lambda$ is contained \textit{properly} in $\mathcal{T}_\epsilon$:\\
For each $T_1,T_2,...T_m$, choose thin strips adjacent to the geodesic sides  and bounded by another geodesic segment ``parallel" to the sides and append them to the rectangles adjacent to the sides. We continue to denote the collection of this slightly thickened rectangles by $R_1,R_2,\ldots R_n$, and their union by $\mathcal{T}_\epsilon$.

\subsection{Approximating $\lambda$ by multicurves}

Let $w_1,w_2,\ldots w_n$ be the \textit{weights} of the train track $\mathcal{T}_\epsilon$, that is, $w_i$ denotes the total transverse measure of the rectangle $R_i$. By our assumption of the maximality of $\lambda$, these weights are all positive reals.\newline

By the above construction of the train track it follows that the sub-arcs $I^\prime_1, \ldots, I^\prime_{3m+1}$ (see (\ref{eq:finer}) in \S 4.1) also have a positive transverse measures $w^\prime_1,\ldots w^\prime_{3m+1}$. This follows from the minimality of $\lambda$: the only way such a sub-arc will carry no measure is if $\lambda$ intersected it only at the endpoints, but the leaf of $\lambda$ passing through an endpoint of one of these sub-arcs is isolated on the complementary side (it is part of the boundary of one of $T_1,\ldots T_m$) and cannot be isolated inside the sub-arc also.\newline

For each $1\leq i\leq n$ the weights
\begin{equation}\label{eq:wprime}
w_i = \sum\limits_{k\in S_i} w^\prime_k
\end{equation}
where $S_i$ is a finite subset of $\{1,2,\ldots 3m+1\}$ as in equation (\ref{eq:finer}).

\begin{defn}
A $3m+1$-tuple of (non-negative) real numbers $(c_1,c_2,\ldots c_{3m+1})$ is an \textit{admissible weighting} of $\mathcal{T}_\epsilon$ if the corresponding weights on the train track given by (\ref{eq:wprime}) satisfy the switch conditions for the train track.
\end{defn}

We begin with the following observations in elementary linear algebra:

\begin{lem}\label{lem:linalg0}
Let $\mathcal{S}$ be a homogeneous system of linear equations in $N$ variables, with all coefficients in the set $\{0,1,-1\}$. Then there exists a constant $C>0$ depending only on $N$ such that for any $N$-tuple  $(x_1,x_2,\ldots,x_N)$ of real numbers that satisfies $\mathcal{S}$ there is an integer $N$-tuple $(k_1,k_2,\ldots,k_N)$ with $\lvert x_i-k_i\rvert <C$ for each $1\leq i\leq N$, which is also a solution.
\end{lem}
\begin{proof}
Since the coefficients of the linear system $\mathcal{S}$ are integers, by Gauss-Jordan elimination there is a basis $\{\vec{v_1},\vec{v_2},\ldots,\vec{v_M}\}$ of the vector space $\mathcal{V}$ of solutions such that each $\vec{v_i}$ is a vector with \textit{rational} entries. Let $L$ be the integer that is the least common multiple of all the denominators of the rational entries, such that $ \vec{w_i} = L\vec{v_i}$ is an integer vector for each $1\leq i\leq M$.  Then this set of $M$ linearly-independent integer  vectors $W$ spans a lattice in $\mathcal{V}$.  Let $C$ be the diameter of the torus $\mathbb{T}^M$ that is a quotient of $\mathcal{V}$ by the action of $W$, or equivalently, the radius of a fundamental domain in $\mathcal{V}$.  Clearly then,  any solution in $\mathcal{V}$ is less than $C$ away from an integer vector.\\
The constant $C$ at this point depends on the linear system $\mathcal{S}$, but notice that since there are $N$ variables, and each coefficient is from the finite set $\{0,1,-1\}$, there are only finitely many possible choices of $\mathcal{S}$ (depending only on $N$), and hence we can choose $C$ to be the maximum value as we vary over all of them.
\end{proof}

\begin{cor}\label{cor:linalg1}
Let $\mathcal{S}$ be a homogeneous system as above, and let $\vec{x}= (x_1,x_2,\ldots,x_N)$ be a solution where each entry is a positive real number. Then there exists a $T_0>0$ such that for any $t>T_0$, there is an integer solution  $(k_1,k_2,\ldots,k_N)$ with each entry positive, such that $\lvert tx_i-k_i\rvert <C$ for each $1\leq i\leq N$.
\end{cor}
\begin{proof}
Note that since $\mathcal{S}$ is homogenous the vector $t\vec{x}$ is also a solution. Each entry of this vector will be greater than $C$ when $t>T_0 = \frac{C}{\min\limits_{1\leq i\leq N} x_i}$. Let $(k_1,k_2,\ldots,k_N)$ be the integer solution close to $t\vec{x}$ that the previous lemma guarantees. Since for each $i$, we have
\begin{center}
$\lvert tx_i-k_i\rvert <C$ and $tx_i>C \implies k_i>0$
\end{center}
we see that each entry of this integer solution is positive.
\end{proof}

\begin{figure}
  \centering
  \includegraphics[scale=0.65]{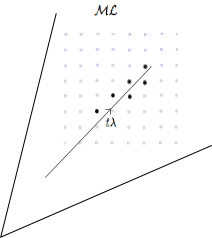}\\
  \caption{The train-track weights give a coordinate chart for measured lamination space. A ray in this convex cone is never far from an integer lattice point (Lemma \ref{lem:linalg}).}
\end{figure}

We now apply this to our setting:

\begin{lem}\label{lem:linalg}
There exists a $C>0$ and $T_0>0$ such that for any $t>T_0$ there is a tuple $\vec{k}=(k_1,k_2,\ldots k_{3m+1})$ of positive integers  such that\newline
(1) $\vec{k}$ is an admissible weighting of $\mathcal{T}_\epsilon$.\newline
(2) $\lvert tw_j^\prime -k_j\rvert < C$ for each $1\leq j\leq 3m+1$.\\
Moreover, $C$ is a constant that is independent of $\epsilon$.
\end{lem}
\begin{proof}
The admissible weights on the train track $\mathcal{T}_\epsilon$ satisfy a linear system $\mathcal{S}$ in $3m+1$ variables corresponding to the switch conditions and equations (\ref{eq:wprime}), which have coefficients in the set $\{0,1,-1\}$. Hence one can apply Corollary \ref{cor:linalg1} with $\vec{x}$ being the positive solution $(w_1,w_2,\ldots w_{3m+1})$ corresponding to the transverse measures of the lamination $\lambda$, and this yields (1) and (2). Also by the lemma, $C$ depends only on $m$, which in turn depends only on the topology of the surface (see \S4.1), and hence is independent of $\epsilon$.
\end{proof}

\begin{defn}
For $t>T_0$, let $\gamma_{t}$ denote the geodesic multicurve corresponding to the admissible integer weighting $\vec{k}$ on the train-track $\mathcal{T}_\epsilon$ satisfying (2) of Lemma \ref{lem:linalg}.
\end{defn}

\begin{lem}\label{lem:gh}
There exists a $T_1>T_0$ such that for any $t>T_1$ the multicurve $\gamma_t$ is contained in $\mathcal{T}_\epsilon \subset X$.
\end{lem}
\begin{proof}
Notice that the induced weights $\bar{k}_1,\ldots\bar{k}_n$ on the {branches} $R_1,\ldots R_n$ of the train track $\mathcal{T}_\epsilon$ (obtained from equation (\ref{eq:wprime})) satisfy:
\begin{equation}\label{eq:wapp}
\lvert tw_i - \bar{k}_i\rvert < (3m+1)C
\end{equation}
for each $1\leq i\leq n$.\newline
This implies that $[\gamma_t]\to [\lambda]$ in $\mathcal{PML}$ as $t\to \infty$, and hence the corresponding geodesic representatives on the surface converge to $\lambda$ in the Hausdorff topology - see \cite{FLP}. (The maximality of $\lambda$ is used here too, since in general it is only true that the supports $\lvert \gamma_t\rvert \to \lvert \lambda^\prime\rvert \supset \lvert \lambda\rvert$.) Since $\lambda$ is a proper subset of the closed set $\mathcal{T}_\epsilon$ (this uses the ``thickening" defined in \S6.2), so is $\gamma_t$ for large enough $t$.
\end{proof}

\subsection{Model rectangles}

Recall that the train-track decomposition of $X$ (described in \S 2.1) persists as we graft along $\lambda$, with the \textit{total width} of the rectangles $R_1,R_2,\ldots,R_n$ increasing along the $\lambda$-grafting ray. We denote the grafted rectangles on $gr_{2\pi t\lambda}X$ by $R_1^t,R_2^t,\ldots R_n^t$.\newline

Here, the \textit{total width} $w_i(t)$  is the maximum width of the rectangle $R_i^t$ in the Thurston metric on $gr_{2\pi t\lambda}X$, and we have
\begin{equation}
\lvert w_i(t) - 2\pi tw_i\rvert <\epsilon
\end{equation}
since the initial \textit{hyperbolic} widths of the rectangles on $X$ is less than $\epsilon$ by Lemma \ref{lem:thinrect}.\newline

We shall use the construction of an almost-conformal \textit{euclidean model rectangle} for a rectangular piece $R^t$ from the collection $\{R_1^t,R_2^t,\ldots,R_n^t\}$ proved in \S 4.3. We restate the results of that section as follows:

\begin{lem}[Lemma \ref{lem:rectModel} and Corollary \ref{cor:rectnoskew}]\label{lem:model}
For any $t>1$, there is a $(1+ C\epsilon)$-quasiconformal map from $R^t$ to a euclidean rectangle of width $2\pi tw_i$ which is  $(\epsilon,\epsilon)$-good on the boundary. (Here $C>0$ is some universal constant.)
\end{lem}

We recapitulate the proof briefly: one first  approximates the uncountable collection of geodesic arcs $R^t\cap \lambda$ by a sequence of \textit{finite, weighted} collections of arcs. For each such finite approximation, one can show that if $t$ is large in proportion to the hyperbolic width the map to the complex plane that \text{straightens} the transverse foliation is an almost-conformal map, and then one takes a limit.

\subsection{Grafted surfaces are close} Let $T_1>0$ be as in Lemma \ref{lem:gh}.
The goal of this section is to prove:

\begin{lem}\label{lem:close}
There exists a (sufficiently large) $T_2>T_1$ such that for any $t>T_2$, we have that $d_\mathcal{T}(gr_{2\pi t\lambda}X,gr_{2\pi \gamma_t}X)< \epsilon$.
\end{lem}

Here's a brief summary of the proof prior to the details:\\

On the initial (ungrafted) surface $X$ one has a train-track $\mathcal{T}_\epsilon$ containing the lamination $\lambda$  which is decomposed into rectangles (corresponding to the branches) by a choice of transverse arc $\tau$. The multi-curve approximation $\gamma_t$ is also contained in $\mathcal{T}_\epsilon$.\\

On the grafted surface $gr_{2\pi t\lambda}X$ the rectangles in the train-track decomposition widen to have more (euclidean) width. Similarly on $gr_{2\pi \gamma_t}X$ the rectangles in the initial decomposition are wider, though not with $C^1$-boundary as the arcs of $\gamma_t$ might intersect $\tau$ at an angle slightly off $\pi/2$.  $\tau$ is then replaced by its ``smoothed" arc which gives the correct rectangle decomposition on $gr_{2\pi \gamma_t}X$.\\

Each rectangle on $gr_{2\pi t\lambda}X$ is then mapped almost-conformally to the corresponding one on  $gr_{2\pi \gamma_t}X$ via their euclidean ``models", by first mapping the boundary and then using the almost-conformal extension lemma (Lemma \ref{lem:ext}). The complement of the train-tracks are isometric as grafting leaves them unaffected, and these put together give the required almost-conformal map between the two surfaces.

\begin{proof}[Proof of Lemma \ref{lem:close}] Since $t>T_1$ we have that $\gamma_t \subset \mathcal{T}_\epsilon$ by Lemma \ref{lem:gh}. We let $R_1^\prime,R_2^\prime,\ldots R_n^\prime$ be the rectangles obtained by grafting $R_1,R_2,\ldots R_n$ along $2\pi \gamma_t$. Note that $\gamma_t \cap R_i$ are a finite collection of geodesic arcs, and grafting along them gives \textit{finitely} grafted rectangles as in \S 6.1.2.\newline

Recall that all horizontal sides of the rectangles $R_1,R_2,\ldots R_n$ (on the surface $X$) lie on the arc $\tau$ which is a segment of a leaf of the horocyclic foliation. We chose the arc $\tau$ to be sufficiently small (see the comment following Lemma \ref{lem:thinrect}) and we can assume that its hyperbolic length $w$ is less than $\epsilon$.\newline

After grafting along $2\pi t\lambda$, this arc is converted to an arc  on $gr_{2\pi t\lambda} X$ of length $w+ 2\pi tw_1 + 2\pi tw_2 + \cdots + 2\pi tw_n$ which we denote by $\tau_\lambda$ (recall that as one grafts, the horocyclic foliation extends to a foliation on the grafted surface). After grafting along $2\pi \gamma_t$, the same arc $\tau$ is converted to an arc $\tau^\prime$ on $gr_{2\pi \gamma_t}X$ of length $w + 2\pi \bar{k}_1 + 2\pi \bar{k}_2 +\cdots + 2\pi \bar{k}_n$ which we can smooth to a $C^1$ arc (see \S 6.1.3) on the grafted surface that we denote by $\tau_\gamma$. This simultaneously ``smooths" the finitely grafted rectangles $R_1^\prime,R_2^\prime,\ldots R_n^\prime$ to a new collection $R_1^{\prime\prime},R_2^{\prime\prime},\ldots R_n^{\prime\prime}$ which we use for the rest of the construction.\newline

Recall also that there are the sub-arcs $J_1,J_2,\ldots J_{3m}$ of $\tau$ that are the horocyclic edges of the complementary regions $T_1,\ldots T_n$. These remain isometrically embedded in the arcs $\tau_\lambda$ and $\tau^\prime$, and also in $\tau_\gamma$ (smoothing of $\tau^\prime$ affects only the segments lying in the grafted part).\newline

\textit{Claim. For sufficiently large $t$ there is an $(\epsilon, 50mC)$-almost-isometry $h$ from $\tau_\gamma$ to $\tau_\lambda$ that restricts to an isometry between the sub-arcs corresponding to $J_1,J_2,\ldots J_{3m}$.}\newline

\textit{Proof of claim.} Recall that the sub-arcs inbetween the $J_1,\ldots, J_{3m}$ are the sub-arcs $I^\prime_1,\ldots I^\prime_{3m+1}$ that had weights $w^\prime_1,\ldots, w^\prime_{3m+1}$ on the surface $X$. On $\tau_\gamma$ and $\tau_\lambda$, the  lengths of these become  $2\pi k_i$ and $2\pi tw^\prime_i$ respectively, which differ by at most $2\pi C$ by (2) of Lemma \ref{lem:linalg}. For large $t$, we have that by Lemma \ref{lem:lem1}, the affine maps between these sub-arcs are $(\epsilon, 2\pi C)$-almost-isometries. By Lemma \ref{lem:lem3} the concatenated map of these together with isometries between the sub-arcs corresponding to $J_1,J_2,\ldots J_{3m}$ give an  $(\epsilon, 7m\cdot 2\pi C)$-almost-isometry from the entire arc $\tau_\gamma$ to $\tau_\lambda$.\qed\newline

By Lemma \ref{lem:thinrect}, for $t$ sufficiently large the height of the finitely grafted rectangle $R_i^\prime$ is sufficiently large and the width is sufficiently small so that one can apply Corollary \ref{cor:smooth1} and obtain, for each $1\leq i\leq n$, an almost-conformal map $f_i$ from $R_i^{\prime\prime}$ to a euclidean rectangle of width $2\pi k_i$ which is $(\epsilon,\epsilon)$-good on the boundary.\newline

Recall from \S6.4 that $R_i^t$ denotes the rectangle $R_i$ on $X$ after grafting along $2\pi t\lambda$. By Lemma \ref{lem:model} for $t$ sufficiently large there exists,  for each $1\leq i\leq n$, an almost-conformal map $g_i$ from $R_i^t$ to a euclidean rectangle of width $2\pi tw_i$ which is $(\epsilon,\epsilon)$-good on the boundary.\newline

From the above claim, the map $h_i$ from the rectangle $\partial S_i$ on $gr_{2\pi \gamma_t}X$ to the rectangle $\partial R_i^t$ on  $gr_{2\pi t\lambda}X$ that is isometric on the vertical sides and restricts to the map $h$ on the horizontal sides (which lie on the arc $\tau_\gamma$) is $(\epsilon,50mC)$-good.\newline

Consider the composition $g_i\vert_\partial \circ h_i \circ f_i\vert_\partial^{-1}$ where $f_i\vert_\partial$ and $g_i\vert_\partial $ are the restrictions of $f_i$ and $g_i$ to the boundary of the rectangles where they are defined. By Lemma \ref{lem:lem2} this is an $(\epsilon, 50mC + 2\epsilon)$-good map between two euclidean rectangles and hence by Lemma \ref{lem:ext} (the height of these rectangles is sufficiently large when $\epsilon$ is sufficiently small) it extends to an almost-conformal map $H_i$ between them. The composition $g_i^{-1}\circ H_i \circ f_i: S_i \to R_i^t$ is an almost conformal map that restricts to $h:\tau_\gamma\to \tau_\lambda$ on the horizontal sides and is isometric on the vertical sides.\newline

The collection of maps $\{H_1,H_2,\ldots H_n\}$ give an almost-conformal map from $S_1\cup \cdots S_n \subset gr_{2\pi \gamma_t}X$ to $R_1^t\cup\cdots R_n^t \subset gr_{2\pi t\lambda}X$ that is isometric on the geodesic sides. Since grafting does not affect the surface $X$ in the complement of $\mathcal{T}_\epsilon$, there is an isometry between $gr_{2\pi \gamma_t}X \setminus S_1\cup \cdots S_n$ and $gr_{2\pi t\lambda}X \setminus R_1^t\cup \cdots R_n^t$. Together with the above collection of maps this isometry defines the almost-conformal map between $gr_{2\pi \gamma_t}X$ and  $gr_{2\pi t\lambda}X$.
\end{proof}

\subsection{Completing the proof}

\begin{proof}[Proof of Proposition \ref{prop:prop1}]

We choose a $t_i>T_2$ (where $T_2$ is as in Lemma \ref{lem:close}) that satisfies equation (\ref{eq:ti}). By the triangle inequality,
\begin{equation*}
d_\mathcal{T} (\pi(gr_{2\pi \gamma_{t_i}}X), Y) \leq d_{\mathcal{T}}(\pi(gr_{2\pi t_i\lambda}(X)),Y) + d_\mathcal{T}(gr_{2\pi t_i\lambda}X,gr_{2\pi \gamma_{t_i}}X) 
\end{equation*}
The first term on the right is less than $\epsilon$ by equation (\ref{eq:ti}) and the second term is $O(\epsilon)$ by Lemma \ref{lem:close}.
\end{proof}

From the discussion at the beginning of \S 6, this proves Proposition \ref{prop:prop0} and hence Theorem 1.4.

\appendix
\section{Proof of Lemma \ref{lem:qclem}}

The purpose of this section is to provide a proof of the following:

\begin{prop}[Lemma \ref{lem:qclem}]
For any $\epsilon>0$ sufficiently small and any $0\leq r\leq \epsilon$ if $f:\mathbb{D}\to \mathbb{D}$ satisfies\newline
(1) $f$ is a quasiconformal map,\newline
(2) The quasiconformal distortion is $(1+ C\epsilon)$ on $\mathbb{D}\setminus B_r$,\newline
then the map $f$ extends to a $(1 + C^\prime\epsilon)$-quasisymmetric map of the boundary, where $C^\prime$ is a constant depending only on $C$.
\end{prop}

For $r=0$ the above result is an easy consequence of the work of Ahlfors and Beurling that we recall as Lemma \ref{lem:AB1} below. \newline

We begin by recalling the definitions and relevant known results in \S A.1, and prove a lemma about moduli of quadrilaterals in \S A.2, from which the proof of the theorem follows. 

\subsection{Background}
A starting point for the rich theory of quasiconformal mappings can be Ahlfors' lectures \cite{Ah1}.\newline

Let $\Gamma$ be a family of (rectifiable) curves in $\mathbb{D}$, and let $\rho$ be a non-negative measurable function on the disk $\mathbb{D}$ satisfying
\begin{equation}\label{eq:f1}
l_\gamma(\rho) = \int\limits_{\gamma} \rho \geq 1
\end{equation}
for all $\gamma \in \Gamma$ and
\begin{equation}\label{eq:f2}
A(\rho) = \iint\limits_\mathbb{D} \rho^2 dz d\bar{z} \neq 0,\infty
\end{equation}
Note that $\rho$ can be thought of as the conformal factor for a metric conformally equivalent to the standard metric on the unit disk.

\begin{defn}\label{defn:ext}
The \textit{extremal} length of $\Gamma$ is denoted by $\lambda(\Gamma)$ is defined as
\begin{equation}\label{eq:def}
\lambda(\Gamma) = \sup\limits_{\rho} A(\rho)^{-1}
\end{equation}
where $\rho$ varies over all non-negative measurable functions satisfying (\ref{eq:f1}) and (\ref{eq:f2}). This is a conformal invariant.
\end{defn}

 \begin{defn}\label{defn:quad}
A \textit{quadrilateral} $Q$ is the unit disk $\mathbb{D}$ together with two disjoint arcs $A$ and $B$ on its boundary. There is a conformal map from $Q$ to a rectangle in $\mathbb{R}^2$ of height $1$ and length $m$  that takes the two arcs to vertical sides. The positive real number $m$ is the \textit{modulus} of the quadrilateral, denoted $mod(Q)$. 
\end{defn}

One of the basic results asserts:

\begin{lem}(Gr\"{o}tzsch)\label{lem:gro} If $\Gamma$ is the collection of all rectifiable curves in $\mathbb{D}$ joining the boundary arcs $A$ and $B$, then $\lambda(\Gamma) = mod(Q)$.
\end{lem}

If $S\subset \mathbb{D}$ is a closed subset containing the boundary $\partial \mathbb{D}$, then we can restrict $\Gamma$ in the lemma above to a collection $\Gamma^\prime$ of curves that are contained in $S$. We denote the corresponding extremal length $\lambda_{S}(Q) = \lambda(\Gamma^\prime)$. We shall also call this the \textit{extremal length restricted to S}. Note that by the above lemma $\lambda_{\mathbb{D}}(Q) = mod(Q)$.\newline
From the definition of extremal length, it is easy to check:

\begin{lem}\label{lem:inc}
If $S\subseteq\mathbb{D}$ then $\lambda_{S}(Q)\geq\lambda_{\mathbb{D}}(Q)$.
\end{lem}

\begin{defn}\label{defn:qc}
Let $\Omega,\Omega^\prime$ be two domains in $\mathbb{C}$. Then a homeomorphism $f:\Omega\to\Omega^\prime$ is said to be $K$-quasiconformal if\newline
(1) $f$ has locally integrable distributional derivatives.\newline
(2) The ratio
\begin{equation}\label{eq:D}
\frac{\lvert f_{z} \rvert - \lvert f_{\bar{z}} \rvert}{\lvert f_{z} \rvert + \lvert f_{\bar{z}} \rvert} \leq K
\end{equation}
almost everywhere in $\Omega$.\newline
The \textit{quasiconformal distortion} at a point in $\Omega$ is defined to be the value of the left hand side of equation (\ref{eq:D}).
\end{defn}

\begin{defn}
A homeomorphism $g:\mathbb{R}\to \mathbb{R}$ is $M$-quasisymmetric if for every $x, t\in \mathbb{R}$, we have
\begin{equation*}
\frac{1}{M}\leq \frac{f(x+t) - f(x)}{f(x) - f(x-t)} \leq M
\end{equation*}
\end{defn}

\begin{defn}
A homeomorphism $g:\partial\mathbb{D}\to \partial\mathbb{D}$ is $M$-quasisymmetric if $h\circ g\circ h^{-1}:\mathbb{H}^2\to \mathbb{H}^2$ is $M$-quasisymmetric when restricted to $\mathbb{R}$, where $h$ is a conformal map from the unit disk $\mathbb{D}$ to the upper half plane $\mathbb{H}^2$.
\end{defn}

The following lemma can be culled from the discussion in section 4 of \cite{AB}.\newline
For $A$,$B$ two intervals in $\mathbb{R}$,  let $\lambda(A,B)$ denote the extremal length of the set of rectifiable paths in $\mathbb{H}^2$ from $A$ to $B$. (To use the above definition of extremal length we first map $\mathbb{H}^2$ conformally to $\mathbb{D}$.)

\begin{lem}\label{lem:qs}
If $g:\mathbb{R}\to\mathbb{R}$ be a homeomorphism such that
\begin{equation*}
\frac{1}{m}\leq \frac{\lambda(g(A),g(B))}{\lambda(A,B)}\leq m
\end{equation*}
for all disjoint intervals $A$, $B$.\newline
Then $g$ is $M$-quasisymmetric, where $M = e^{A(m-1)}$ where $A\approx 0.228$ is a universal constant.
\end{lem}
The following are the fundamental results of Ahlfors and Beurling (\cite{AB}) (for the version stated here see \cite{Bish}). Briefly, quasisymmetric maps of the boundary circle extend to quasiconformal maps of the unit disk, and vice versa, with the distortion constants ($K$ and $M$ as above) being close to $1$ if one of them is.

\begin{lem}\label{lem:AB1}
For any $K>1$ there is an $M>1$ such that if $f:\mathbb{D}\to \mathbb{D}$ is a $K$-quasiconformal map then it extends to an $M$-quasisymmetric homeomorphism of the boundary. Moreover, there is a $K_0>1$ and $C_0<\infty$ such that if $K= 1+ \epsilon < K_0$ then we can take $M\leq 1 + C_0\epsilon$.
\end{lem}

\begin{lem}\label{lem:AB2}
Any $M$-quasisymmetric homeomorphism of $\partial\mathbb{D}$ can be extended to a $K$-quasiconformal map of $\mathbb{D}$.  Moreover, there is a $M_1>1$ and $C_1<\infty$ such that if $M= 1+ \epsilon < M_1$ then we can take $K\leq 1 + C_1\epsilon$.
\end{lem}

Finally, we note the following standard consequence of quasiconformality (eg, see Chapter II of \cite{Ah1}):

\begin{lem}\label{lem:qc}
Let $\Omega,\Omega^\prime \subset \mathbb{D}$ be domains. If $f:\Omega\to\Omega^\prime$ is a $K$-quasiconformal map, then for any collection $\Gamma$ of rectifiable curves in $\Omega$, we have
\begin{equation}
\frac{1}{K}\leq \frac{\lambda_{\Omega^\prime}(f(\Gamma))}{\lambda_{\Omega}(\Gamma)}\leq K
\end{equation}
\end{lem}

\subsection{An extremal length lemma}

Let $A,B\subset \mathbb{D}$ be two disjoint boundary arcs,  and $\Gamma$ the collection of rectifiable paths from $A$ to $B$ and let $E=B_r$ be the ball of radius $r<1$.\newline

In this section we show (Lemma \ref{lem:a14}) that the extremal length of $\Gamma$ changes by a multiplicative factor of $1 + O(r)$ when $E$ is excised, that is, when we restrict to the family of curves joining $A$ and $B$ and avoiding $E$. The constant in the $O(r)$ term is independent of the arcs $A$, $B$.\newline

The following consequence of the Ko\"{e}be distortion theorem is used in its proof.
\begin{lem}\label{lem:conf}
Let $E=B_r\subset \mathbb{D}$ and let $\phi:\mathbb{D}\to \mathbb{C}$ be a conformal embedding such that $\phi(0)=0$. Then
\begin{equation}\label{eq:p0}
diam(\phi(E)) < Crdist(\phi(E),\partial \phi(\mathbb{D}))
\end{equation}
for all $r$ sufficiently small, for some universal constant $C$.
\end{lem}
\begin{proof}
Let $\delta = dist(0,\partial \phi(\mathbb{D}))$.\newline
By a consequence of the Ko\"{e}be distortion theorem (see Corollary 1.4 of \cite{Pom}) we have
\begin{equation}\label{eq:p1}
\lvert \phi^\prime(0) \rvert \leq 4\delta
\end{equation}
and by the other direction of the distortion theorem (Theorem 1.3 of \cite{Pom}) we have
\begin{equation}
\lvert \phi(z) \rvert \leq \lvert \phi^\prime(0) \rvert \frac{\lvert z\rvert}{(1-\lvert z\rvert)^2} 
\end{equation}
which using (\ref{eq:p1}) gives
\begin{equation}\label{eq:p4}
\lvert \phi(z) \rvert < 4\delta r/(1-r)^2 < 8r\delta
\end{equation}
for any $z\in E$, since then $\lvert z\rvert \leq r$ and $r$ is sufficiently small.\newline
Hence
\begin{equation}\label{eq:p2}
diam(\phi(E)) < 16r\delta
\end{equation}

Now for $\omega\in E$, let $ d_w= dist(\phi(\omega),\partial \phi(\mathbb{D})) = \lvert \phi(\omega) - \phi(s) \rvert$ for some $\phi(s)\in \partial \phi(\mathbb{D})$.\newline
We have
\begin{equation*}
\delta \leq \lvert \phi(s)\rvert \leq d_w + \lvert\phi(\omega)\rvert \leq d_w + 8r\delta
\end{equation*}
where the last inequality is by (\ref{eq:p4}) and the second the triangle inequality.\newline
Rearranging, and taking an infimum over $\omega \in E$ on the left hand side, we obtain
\begin{equation}
\delta(1-8r) \leq dist(\phi(E), \partial \phi(\mathbb{D}))
\end{equation}
which implies, for $r$ sufficiently small,
\begin{equation}\label{eq:p3}
\delta \leq (1 + C^\prime r)dist(\phi(E), \partial \phi(\mathbb{D}))
\end{equation}
for some constant $C^\prime$. The proof is complete on combining (\ref{eq:p2}) and (\ref{eq:p3}).
\end{proof}

\begin{lem}\label{lem:a14}
Let $E=B_r\subset \mathbb{D}$ and $Q$ the quadrilateral defined by the two boundary arcs $A$, $B$ as above. Then 
\begin{equation}\label{eq:rat2}
1 \leq \frac{\lambda_{\mathbb{D}\setminus E}(Q)}{\lambda_\mathbb{D}(Q)} \leq 1 + C^\prime r
\end{equation}
where $C^\prime$ is a constant  independent of $Q$.
\end{lem}
\begin{proof}
The first inequality of (\ref{eq:rat2}) follows from Lemma \ref{lem:inc} so we need only to prove the other inequality.\newline

Let $\phi$ be a conformal map from $\mathbb{D}$ to a rectangle $R$  of vertical height $1$ and horizontal length $m = mod(Q)$ such that the arcs $A$ and $B$ are taken to the left and right vertical sides respectively. We further require that $\phi(0)=0$. Such a map can be defined using elliptic integrals (see for example Chapter III of \cite{Ah1}).\newline
It is well-known (see Definition \ref{defn:quad} and Lemma \ref{lem:gro}) that this conformal domain realizes the extremal length of $Q$: the conformal metric $\rho\equiv 1$ on the rectangle pulled back via $\phi$ realizes the supremum in Definition \ref{defn:ext}.\newline

Let $\Gamma$ be the set of all rectifiable paths in $R$  between the vertical sides , and let $\Gamma^\prime$ be the subcollection of $\Gamma$ of paths disjoint from $\phi(E)$.\newline
We shall adapt the Gr\"{o}tzsch argument to show that $\rho$ is close to being extremal for the collection $\Gamma^\prime$.\newline

Let $S$ be a strip $S=[0,m]\times J$ of vertical ($y$-) height $\lvert J\rvert = diam(\phi(E))$ and horizontal($x$-) range $m$ that contains $\phi(E)$. By Lemma \ref{lem:conf}, we know that
\begin{equation}\label{eq:g2}
diam(\phi(E))< Cr
\end{equation}
for $r$ small, since $\phi(E)\subset R$ implies $dist(\phi(E),\partial \phi(\mathbb{D})) \leq min\{1,m\}\leq 1$.  \newline

Let $\rho^\prime$ be a conformal factor for $R$ that satisfies 
\begin{equation*}
l_\gamma(\rho^\prime)\geq 1 
\end{equation*}
for all $\gamma\in \Gamma^\prime$.\newline
In particular,
\begin{equation}\label{eq:g1}
\int\limits_{0}^{m}\rho^\prime(x,y)dx \geq 1
\end{equation}
for any $y$ in $[0,1]\setminus J$.\newline
By integrating (\ref{eq:g1}) over $y$ ranging over $[0,1]\setminus J$, we get
\begin{equation*}
1-Cr \leq \int\limits_{[0,1]\setminus J} 1 dy \leq \int\limits_{[0,1]\setminus J}\int\limits_{0}^{m} \rho^\prime(x,y)dxdy = \iint\limits_{R\setminus S}  \rho^\prime(x,y)dxdy
\end{equation*}
Squaring, and using the Cauchy-Schwarz inequality for the right hand term, we get
\begin{equation*}
(1-Cr)^2\leq \iint\limits_{R\setminus S}  (\rho^\prime)^2dxdy \iint\limits_{R\setminus S}  1^2dxdy \leq(\iint\limits_{R}  (\rho^\prime)^2dxdy)m
\end{equation*}
So
\begin{equation}\label{eq:g3}
 (\iint\limits_{R} (\rho^\prime)^2 dxdy )^{-1}\leq m/(1-Cr)^2 \leq m(1 + O(r))
 \end{equation}
 for $r$ sufficiently small.\newline
Taking a supremum over $\rho^\prime$ as in Definition \ref{defn:ext} we get
\begin{equation*}
\lambda_{\mathbb{D}\setminus E}(Q) = \lambda_{R\setminus \phi(E)}(Q) = \lambda(\Gamma^\prime) \leq m(1+O(r))
\end{equation*}
where the first equality holds since extremal length is a conformal invariant. Since $m = \lambda_{\mathbb{D}}(Q)$, this is the right hand equality of (\ref{eq:rat2}) and the proof is complete.
\end{proof}

\subsection{Proof of proposition}
Henceforth let $f:\mathbb{D}\to\mathbb{D}$ be the quasiconformal map with quasiconformal distortion $1 +C\epsilon$ off a small ball $B_r$, as in Proposition A.1.\newline

We need the following "quasiconformal" version of Lemma \ref{lem:conf} for maps of the unit disk:\begin{lem}\label{lem:qconf}
$diam(f(B_r)) =O(r^{1-\epsilon})$.
\end{lem}
\begin{proof} 
Let $d = diam(f(B_r))$. For convenience we shall assume $C=1$.\newline
It is well-known (eg. see III.A of \cite{Ah1}) that for an annular domain on the plane that contains $0,1$ in the bounded component of its complement, and the interval $[c,\infty)$ in the unbounded component where $c>1$, we have
\begin{equation}\label{eq:ann}
\lambda(P) < \frac{1}{2\pi}\ln 16c
\end{equation}
where $P$ is the set of rectifiable curves connecting the inner boundary component to the outer boundary component.\newline

Since $f$ is a homeomorphism, if  $A= \mathbb{D}\setminus B_r$ and $\Gamma$ the set of paths between the boundary components of $A$ then $f(A)$  is topologically an annulus, and by rotation and scaling we see that we see that it satisfies the above condition with  $c = 1/d$. So we have
\begin{equation}\label{eq:e1}
\lambda(f(\Gamma)) < \frac{1}{2\pi}\ln \frac{16}{d}
\end{equation}
By Lemma \ref{lem:qc} and the fact that $f$ is $(1 +\epsilon)$-quasiconformal on $\mathbb{D}\setminus B_r$ we know that
\begin{equation}\label{eq:e2}
(1-\epsilon)\lambda(\Gamma) \leq \frac{1}{1+\epsilon}\lambda(\Gamma) \leq \lambda(f(\Gamma))
\end{equation}
but since
\begin{equation*}
\lambda(\Gamma) = \frac{1}{2\pi}\ln \frac{1}{r}
\end{equation*}
we obtain from (\ref{eq:e1}) and (\ref{eq:e2}) that
\begin{equation*}
d \leq 16r^{1-\epsilon}
\end{equation*}
\end{proof}
\begin{cor}\label{cor:qconf} If $r\leq \epsilon$ then $diam(f(B_r))=O(\epsilon)$.\end{cor}
\begin{proof} The maximum of $x^{-x}$ is $e^{1/e}\approx 1.44$. So $r^{1-\epsilon}\leq \epsilon^{1-\epsilon}< 1.45\epsilon$.
\end{proof}

\begin{figure}
  \centering
  \includegraphics[scale=0.45]{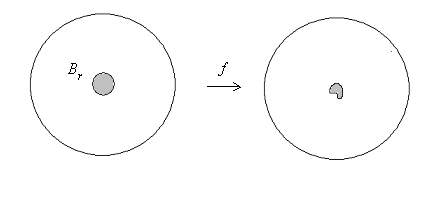}\\
  \caption{The image of $B_r$ is of small diameter. By Lemma \ref{lem:a14}, the extremal lengths change by a small factor when restricted to the complement of the  shaded regions. }
\end{figure}

\begin{proof}[Proof of Proposition 1.1]
By Lemma \ref{lem:qs} is enough to show that for any pair of disjoint arcs $A$ and $B$ on $\partial D$, we have
\begin{equation}\label{eq:goal}
1- O(\epsilon)\leq \frac{\lambda_{\mathbb{D}}(f(Q))}{\lambda_{\mathbb{D}}(Q)}\leq 1+ O(\epsilon)
\end{equation}
where $Q$ is the corresponding quadrilateral, and the constant in $O(\epsilon)$ is independent of $Q$.\newline
This is because in the notation of Lemma \ref{lem:qs}, if $m = 1 + O(\epsilon)$ then $M=e^{A(m-1)} =1+ O(\epsilon)$.\newline

Let $r\leq \epsilon$. By Lemma \ref{lem:a14} we have
\begin{equation*}
1 \leq \frac{\lambda_{\mathbb{D}\setminus B_r}(Q)}{\lambda_\mathbb{D}(Q)} \leq 1 + O(\epsilon)
\end{equation*}
and by Corollary \ref{cor:qconf} and Lemma \ref{lem:a14} we have
\begin{equation*}
1 \leq \frac{\lambda_{\mathbb{D}\setminus f(B_r)}(f(Q))}{\lambda_\mathbb{D}(f(Q))} \leq 1 + O(\epsilon)
\end{equation*}
Now by Lemma \ref{lem:qc}, since $f$ is almost-conformal on $\mathbb{D}\setminus B_r$, we have
\begin{equation*}
1- O(\epsilon) \leq \frac{\lambda_{\mathbb{D}\setminus f(B_r)}(f(Q))}{\lambda_{\mathbb{D}\setminus B_r}(Q)} \leq 1 + O(\epsilon)
\end{equation*}
The required (\ref{eq:goal}) follows easily from the above three inequalities.
\end{proof}

\bibliographystyle{amsalpha}
\bibliography{qcref}

\end{document}